\documentclass[12pt,reqno]{amsart}

\usepackage{mathrsfs}
\usepackage{amsmath}
\usepackage{amssymb}
\usepackage{amsfonts}
\usepackage{amsthm}
\usepackage{a4wide}
\usepackage{graphicx}
\usepackage{color}
\usepackage[sans]{dsfont}
\usepackage{linearA}
\usepackage{pgfkeys}
\usepackage{longtable}
\usepackage{pdflscape}
\usepackage{float}
\usepackage{cancel}
\usepackage{dsfont}
\usepackage{setspace}
\usepackage{array}
\usepackage{hyperref}
\hypersetup{colorlinks=true,allcolors=[rgb]{0,0,0.6}}
\usepackage[english]{babel}
\usepackage[applemac]{inputenc}
\usepackage[T1]{fontenc}
\usepackage{tikz,pgflibraryshapes}
\usepackage{enumitem}
\usepackage{eufrak}
\usepackage[square,numbers,sort&compress]{natbib}
\usepackage[shadow,textwidth=22mm,textsize=tiny]{todonotes}
\usepackage[left=2.5cm,right=2.5cm,top=3cm,bottom=3cm]{geometry}
\usepackage{soul}

\newcommand{\N}{\mathbb{N}}
\newcommand{\C}{\mathbb{C}}
\newcommand{\R}{\mathbb{R}}
\newcommand{\F}{\mathfrak{F}_s}
\newcommand{\FF}{\mathcal{F}}
\newcommand{\Real}{\mathfrak{Re}}

\newcommand{\HHH}{\mathcal{H}}

\newcommand{\I}{\mathbf{I}}

\newcommand{\ZZ}{\mathcal{Z}}
\newcommand{\LL}{\mathcal{L}}

\newcommand{\E}{\mathcal{E}}

\newcommand{\Proj}{\Pi_V^\perp}

\newcommand{\abs}[1]{\left|#1\right|}
\newcommand{\norm}[1]{\left\|#1\right\|}
\newcommand{\bra}[1]{\left\langle #1 \right|}
\newcommand{\ket}[1]{\left| #1 \right\rangle}

\newcommand{\ps}[2]{\left\langle #1,#2 \right\rangle}

\newcommand{\diff}{\mathop{}\!\mathrm{d}}

\newtheorem{theorem}{Theorem}[section]
\newtheorem{lemma}[theorem]{Lemma}
\newtheorem{proposition}[theorem]{Proposition}

\newtheorem{remark}[theorem]{Remark}
\newtheorem{assumption}{Hypothesis}

\numberwithin{equation}{section}

\author[S. Breteaux]{S{\'e}bastien Breteaux}
\address[S. Breteaux]{Universit{\'e} de Lorraine, CNRS, IECL, F-57000 Metz, France}
\email{sebastien.breteaux@univ-lorraine.fr}

\author[J. Faupin]{J{\'e}r{\'e}my Faupin}
\address[J. Faupin]{Universit{\'e} de Lorraine, CNRS, IECL, F-57000 Metz, France}
\email{jeremy.faupin@univ-lorraine.fr}

\author[J. Payet]{Jimmy Payet}
\address[J. Payet]{Universit{\'e} de Lorraine, CNRS, IECL, F-57000 Metz, France}
\email{jimmy.payet@univ-lorraine.fr}

\theoremstyle{plain}

\theoremstyle{plain}
\theoremstyle{plain}

\theoremstyle{plain}

\theoremstyle{remark}

\makeatother

\usepackage{babel}

\providecommand{\lemmaname}{Lemma}
\providecommand{\notationname}{Notation}
\providecommand{\propositionname}{Proposition}
\providecommand{\theoremname}{Theorem}

\begin{document}
 \title[Quasi-Classical Ground States]{Quasi-Classical Ground States. II. \\ Standard Model of Non-Relativistic QED}

\begin{abstract}
We consider a non-relativistic electron bound by an external potential and coupled to the quantized electromagnetic field in the standard model of non-relativistic QED. We compute the energy functional of product states of the form $u\otimes \Psi_f$, where $u$ is a normalized state for the electron and $\Psi_f$ is a coherent state in Fock space for the photon field. The minimization of this functional yields a Maxwell--Schr\"odinger system up to a trivial renormalization. We prove the existence of a ground state under general conditions on the external potential and the coupling. In particular, neither an ultraviolet cutoff nor an infrared cutoff needs to be imposed. Our results provide the convergence in the ultraviolet limit and the second-order asymptotic expansion in the coupling constant of the ground state energy of Maxwell--Schr\"odinger systems.
\end{abstract}
\maketitle
\tableofcontents
\section{Introduction}
We consider in this paper a non-relativistic spin-$\frac{1}{2}$ particle (an electron)  minimally coupled to the quantized radiation field in the standard model of non-relativistic quantum electrodynamics, with an external potential~$V$. This physical system is mathematically described by a Pauli-Fierz Hamiltonian~$\mathbb{H}$, introduced in~\cite{PauliFierz38}, whose spectral and scattering theories have been thoroughly studied since the end of the nineties (see, among others, \cite{BachFrohlichSigal98,BachFrohlichSigal99,DerezinskiGerard99,GriesemerLiebLoss01,Sigal09, Spohn04,HaslerHerbst11, AraiHirokawa97, Gerard00, GriesemerHasler09} and references therein). 
To be well-defined, the Pauli-Fierz Hamiltonian~$\mathbb{H}$ requires an unphysical regularization: the interaction term comes with an ultraviolet cutoff. Finding a renormalization procedure leading to the definition of the model in the ultraviolet limit remains an important open problem.

Restricting the energy functional associated to $\mathbb{H}$ to well-chosen classes of states allows one to study the energy and its infimum more easily. In the translation invariant case ($V=0$), considering the set of general product states $u\otimes \Psi$ where~the state $u$ of the electron is a unit vector in the Hilbert space~$\HHH_\mathrm{el}=L^2(\R^3;\C^2)$ and the state~$\Psi$ of the photon field is a unit vector in Fock space, the ultraviolet divergence of the infimum of the energy functional~$\langle(u\otimes\Psi),\mathbb{H}(u\otimes\Psi)\rangle$ has been studied by Lieb and Loss in \cite{LiebLoss99}, and by Bach and Hach in \cite{BachHach20}. Denoting by $\Lambda$ the ultraviolet parameter associated to the ultraviolet cutoff introduced into the interaction Hamiltonian, it is shown in \cite{LiebLoss99,BachHach20} that the corresponding ground state energy diverges as~$\Lambda^{12/7}$ in the ultraviolet limit. Also in the translation invariant case, at a fixed total momentum, the existence and uniqueness of a minimizer of the energy functional over coherent or quasifree states has been studied in~\cite{BachBreteauxTzaneteas13}.

Product states of the form~$u\otimes \Psi_{\vec{f}}$, with~$\Psi_{\vec{f}}$ a coherent state parametrized by vectors~$\vec{f}$ in the one-particle Hilbert space~$\mathfrak{h}$ for the field,  have been considered in~\cite{CorreggiFalconi18, CorreggiFalconiOlivieri19_01, CorreggiFalconiOlivieri19_02,CorreggiFalconiOlivieri20}. 
The energy functional 
\begin{equation}\label{eq:quasi-cl}
(u,\vec{f})\mapsto\langle (u\otimes\Psi_{\vec{f}}),\mathbb{H}(u\otimes\Psi_{\vec{f}})\rangle
\end{equation}
 is then called the \emph{quasi-classical energy}. Indeed, assuming that the field degrees of freedom are `almost classical', in the sense that the creation and annihilation operators $a^*$, $a$ 
are rescaled as $a^*_\varepsilon=\sqrt{\varepsilon}a^*$, $a_\varepsilon=\sqrt{\varepsilon}a$ (see also~\cite{AmmariNier08}),
it is shown in~\cite{CorreggiFalconi18, CorreggiFalconiOlivieri19_01, CorreggiFalconiOlivieri19_02,CorreggiFalconiOlivieri20}, under suitable assumptions, that the ground state energy of the rescaled Pauli-Fierz Hamiltonian $\mathbb{H}_{\varepsilon}$ converges to the infimum of the quasi-classical energy functional as~$\varepsilon\to0$.

In this paper, we also consider the quasi-classical energy functional \eqref{eq:quasi-cl}. Up to a trivial renormalization, we will see that minimizing \eqref{eq:quasi-cl} boils down to minimizing $\mathcal{E}_V(u,\vec{A}_{\vec{f}})$, for some~$\vec{f}$-dependent magnetic potential~$\vec{A}_{\vec{f}}$, where  $\mathcal{E}_V(u,\vec{A})$  is the Maxwell-Schr\"odinger energy in the Coulomb gauge, given by 
\begin{align}\label{eq:comput_E(u,A_intro}
\mathcal{E}_V(u,\vec{A}) =&
\| \vec{\sigma} \cdot (-i\vec{\nabla} - g\hat\chi * \vec{A}) \, u \|^2_{L^{2}} + \langle u, V u\rangle_{L^2}
+\frac{1}{32\pi^3} \| \vec A \|^2_{\dot{H}^1}\,.
\end{align}
Here $\vec{\sigma}$ is the vector of Pauli matrices, $g$ is a coupling constant and $\chi$ a coupling function. The coefficient $(32\pi^3)^{-1}$ comes from our choice of normalization of the Fourier transform, see below.

For a general class of external potentials~$V$ (including both binding and confining potentials) and coupling functions $\chi$, we prove the existence of a 
ground state for~$\mathcal{E}_V$. In particular, neither an infrared nor an ultraviolet cutoff is needed in the interaction term of the energy. Furthermore, if an ultraviolet cutoff of parameter $\Lambda$ is imposed, our results show that the ground state energy converges in $\mathbb{R}$, as $\Lambda\to\infty$. 

To prove the existence of a quasi-classical ground state, we follow the usual strategy of the calculus of variations. The main difficulty comes from the possible absence of an ultraviolet cutoff. This induces singular terms with a critical behavior in the energy functional that we handle using suitable estimates in Lorentz spaces. Note that Kramer's symmetry of the Maxwell-Schr\"odinger energy functional implies that the minimizer is not unique (even up to a phase in $u$).

In~\cite{FrohlichLiebLoss86}, Fr\"ohlich, Lieb and Loss studied the minimization problem of similar energy functionals. Compared to~\cite{FrohlichLiebLoss86}, our results provide the existence of a ground state for large classes of external potentials and coupling terms, and allow us to pass to the ultraviolet limit. Moreover, we compute the second order asymptotic expansion at small coupling of the ground state energy.

In the companion paper~\cite{BreteauxFaupinPayet22_1}, we study the same problem in the case of a spinless, non-relativistic particle linearly coupled to a scalar, quantized radiation field. Although the overall strategies in~\cite{BreteauxFaupinPayet22_1} and the present paper are similar, the arguments used in the proofs are significantly different. In particular, in the case of linear coupling, an easy argument shows that the minimization of the quasi-classical ground state energy reduces to the minimization of the Hartree energy (over the state $u$ of the non-relativistic particle). In the present context such a simplification does not occur: We have to minimize \eqref{eq:comput_E(u,A_intro} over $(u,\vec{A})$ in suitable spaces, with the constraint $\|u\|_{L^2}=1$ for the electron state and no constraint on the divergence-free vector potential $\vec{A}$ in $\dot{H}^1$. Note however that some technical results concerning the electronic Hamiltonian are used both in~\cite{BreteauxFaupinPayet22_1} and in this paper. They are stated here without proof.

We have focused in this work on the static problem, but the dynamical version, the Maxwell--Schr\"odinger equations, has of course been also largely studied in the literature. In particular, the Maxwell--Schr\"odinger equations have been derived in~\cite{NakamitsuTsutsumi86}, where results on the existence of solutions have been proven. The dynamics of the Maxwell--Schr\"odinger equations has been further studied in~\cite{Tsutsumi93, GuoNakamitsuStrauss95, GinibreVelo06, NakamuraWada07, Wada12, BenciFortunato14, ChupengCao18, LiuWada, Kieffer20, ColinWatanabe20, Shimomura03,BejenaruTataru09}. 
In relation with many-body systems, the Maxwell--Schr\"odinger equations have been obtained from many-body dynamics in~\cite{LeopoldPickl20}, see also~\cite{CorreggiFalconiOlivieri19_02}.

\medskip

\noindent \emph{Notations.}
We recall that for $1\le p<\infty$, the Lorentz spaces (or weak $L^p$ spaces) $L^{p,\infty}(\mathbb{R}^3)$ are defined as the set of (equivalence classes of) measurable functions $f:\mathbb{R}^3\to\mathbb{C}$ such that
\begin{equation}\label{eq:norm_weakLp}
\|f\|_{L^{p,\infty}}:=\sup_{t>0}  \lambda \big( \{ |f|>t\}\big)^{\frac1p}t,
\end{equation}
is finite, where $\lambda$ denotes Lebesgue's measure.

The Fourier transform acting on tempered distribution is denoted by $\mathcal{F}$, its inverse being given by~$(2\pi)^{-3}\bar{\mathcal F}$. (We use the normalization $\mathcal F(f)(x)=\int_{\mathbb{R}^3} e^{-ix\cdot\xi} f(\xi) \diff \xi$ for $f$ in~$L^1(\mathbb{R}^3)$, and hence~$\bar{\mathcal F}(f)(x)=\int_{\mathbb{R}^3} e^{ix\cdot\xi} f(\xi) \diff \xi$. This normalization is not the standard one but it will be convenient in our context.) Throughout the paper, we use the following convention about the convolution product. Let~$f$ and~$g$ be functions associated to tempered distributions. Assume that $\FF(g)$ identifies with a function such that $f\FF(g)$ can be associated to a tempered distribution. We write
\begin{equation}\label{eq:convfourier}
\FF(f)*g := (2\pi)^{-3} \FF( f \bar{\FF}(g) ).
\end{equation}
This convention is convenient in our context. It extends the well-known equality which holds e.g. if~$f$ and~$g$ are in~$L^1$ or $f$ is in~$L^2$ and~$g$ in~$L^1$.

In several places, we use localization functions $\eta$ and  $\tilde{\eta}$ in $\mathrm{C}^\infty(\mathbb{R}^3)$ such that~%$0\le\eta\le1$, 
$\eta(x)=1$ if~$|x|\le1$, $\eta(x)=0$ if $|x|\ge2$ and 
\begin{equation*}
\eta^2+\tilde{\eta}^2=1\,.
\end{equation*}
For all $R>0$, we set 
\begin{equation}\label{eq:defetaR}
\eta_R(x):=\eta(x/R) \quad\text{and} \quad\tilde{\eta}_R(x):=\tilde{\eta}(x/R)\,.
\end{equation}

If $\mathcal{H}_1$, $\mathcal{H}_2$ are two Hilbert spaces, $\mathcal{L}(\mathcal{H}_1,\mathcal{H}_2)$ stands for the set of bounded linear operators from $\mathcal{H}_1$ to $\mathcal{H}_2$. Given a linear operator $A$ on a Hilbert space $\mathcal{H}$, we denote by $\mathcal{D}(A)$ its domain and $\mathcal{Q}(A)$ its form domain. The topological dual of a Banach space $\mathcal{B}$ is denoted by~$\mathcal{B}^*$.

\subsection{The electronic Hamiltonian}\label{subsec:elec}

If the coupling between the electron and the photon field is turned out, the free Hamiltonian for the electron is of the form
\begin{equation*}
\begin{pmatrix}
H_V & 0 \\
0 & H_V
\end{pmatrix}
\quad\text{on}\quad \mathcal{H}_{\text{el}}:=L^2(\mathbb{R}^3;\mathbb{C}^2)=L^2(\mathbb{R}^3;\mathbb{C})\oplus L^2(\mathbb{R}^3;\mathbb{C}),
\end{equation*}
where
\begin{equation}\label{eq:HV}
H_V := -\Delta+V(x)
\end{equation}
is defined on a domain contained in~$L^2(\mathbb{R}^3;\mathbb{C})$. Here $V:\mathbb{R}^3\to\mathbb{R}$ is the external potential. We display the dependence on $V$ since one of our main hypotheses (see Hypothesis \ref{condVPauliFierz}) assumes the existence of a decomposition $V=V_1+V_2$ such that $V_1\ge0$, $V_2$ vanishes at $\infty$ and there is a gap between the ground state energies of $H_V$ and $H_{V_1}$.

The main examples we have in mind are confining potentials, $V(x)\to\infty$ as $|x|\to\infty$, and Coulomb-type potentials, $V(x)=-c|x|^{-1}$ with $c>0$. We introduce general hypotheses on~$V$ that are fulfilled by a large class of potentials, including the two preceding examples. As we will see below, some of our main results have interesting consequences in special cases, especially when $V$ is confining.

We set
\begin{equation*}
\mu_V:=\inf\sigma(H_{V}), 
\end{equation*}
and likewise if $V$ is replaced by another potential. For $U:\mathbb{R}^3\to\mathbb{R}$, we denote by
\begin{equation*}
U_+:=\max(U,0),\quad U_-:=\max(-U,0),
\end{equation*}
the positive and negative parts of $U$, respectively, so that $U=U_+-U_-$. 

We make the following hypothesis.
\begin{assumption}[Conditions on $V$]\label{condVPauliFierz}
The potential $V$ satisfies $V(x)=V(-x)$ for all $x$ in $\mathbb{R}^3$ and there exist $a\ge0$ and~$b$ in~$\mathbb{R}$ such that 
\begin{equation*}
V_{-}\leq a\sqrt{-\Delta}+b
\end{equation*}
in the sense of quadratic forms on $H^{1/2}(\mathbb{R}^3)$. Moreover, $V$ decomposes as $V=V_{1}+V_{2}$ with
\begin{enumerate}[label=(\roman*)]
\item $V_{1}\in L_{\mathrm{loc}}^{1}(\mathbb{R}^{3};\mathbb{R}^{+})$,
\item $V_{2}\in L_{\mathrm{loc}}^{3/2}(\mathbb{R}^{3};\mathbb{R})$ and ${\displaystyle \lim_{|x|\to\infty}V_{2}(x)=0}$.
\end{enumerate}
\end{assumption}

Since $V_+\ge0$, $H_{V_+}=-\Delta+V_+$ identifies with a non-negative self-adjoint operator on $L^2(\mathbb{R}^3)$ with form domain
\begin{equation*}
\mathcal{Q}(H_{V_+}) = \mathcal{Q}(-\Delta)\cap \mathcal{Q}(V_+)=\Big \{ u \in H^1(\mathbb{R}^3) , \, \int_{\mathbb{R}^3} V_+(x) |u(x) |^2 \diff x < + \infty \Big \}.
\end{equation*} 
Moreover, it follows from Hypothesis \ref{condVPauliFierz} that $H_V$ identifies with a semi-bounded self-adjoint operator with form domain $\mathcal{Q}(H_V) = \mathcal{Q}(H_{V_+}) = \mathcal{Q}(H_{V_1})$. In particular, $\mu_V$ and $\mu_{V_1}$ are well-defined. See Section~\ref{subsec:Estimate-elec} for justifications. 

The state of the electron is represented by a unit vector in the space $L^2(\mathbb R^3;\mathbb{C}^2)$. We set
\begin{equation}\label{eq:QV_intro}
\mathcal{Q}_V:=\mathcal{Q}(H_V)\otimes \mathbb{C}^2\, ,
\end{equation}
and note that $\mathcal{Q}_V$ is a Hilbert space for the norm
\begin{equation*}%\label{eq:normQbis}
\|u\|^2_{\mathcal{Q}_V}:=\|u\|^2_{H^1}+\big\|(V_+)^{\frac12}\otimes  \mathbf{I}_{\mathbb{C}^2}\, u\big\|^2_{L^2} \, .
\end{equation*}
We will most of the time consider an electron state $u$ in
\begin{equation}\label{eq:defU_intro}
\mathcal{U} := \{u\in\mathcal{Q}_V\mid \|u\|_{L^2}=1\} \,.
\end{equation}

Finally, in order to obtain the asymptotic expansion of the infimum of the Maxwell-Schr\"odinger  energy functional with respect to the coupling constant, we will require that  $H_V$ has a unique ground state.  By Perron-Frobenius arguments, it is well-known that, under suitable conditions on $V$, if $\mu_V$ is an eigenvalue of $H_V$ then it is simple and there exists a corresponding strictly positive eigenstate
 (see e.g. \cite[Theorems XIII.46 and XIII.48]{ReedSimonII}). We will make the following related hypothesis.
\begin{assumption}[Ground state of $H_V$]\label{condGS} The ground state energy $\mu_V$ of~$H_V=-\Delta+V$ is a simple isolated eigenvalue associated to a unique positive ground state~$u_V$ belonging to~$L^2(\mathbb{R}^3;\mathbb{R}_+)$ and such that $\|u_V\|_{L^2}=1$.
\end{assumption}

The orthogonal projection onto the vector space spanned by $\left(\begin{smallmatrix}u_V\\0\end{smallmatrix}\right)$ and $\left(\begin{smallmatrix}0\\u_V\end{smallmatrix}\right)$ in $L^2(\mathbb{R}^3;\mathbb{C}^2)$ is denoted by~$\Pi_V$. We also set $\Pi_V^\perp:=\mathbf{I}_{L^2(\mathbb{R}^3;\mathbb{C}^2)}-\Pi_V$.

\subsection{Standard model of non-relativistic QED}\label{subsec:SM}

In the standard model of non-relativistic QED, the quantized electromagnetic field
is represented by a vector-valued bosonic field whose Hilbert
space is given by the symmetric Fock space
\begin{equation*}
\mathcal{H}_{\text{f}}:=\mathfrak{F}_{s}(L^{2}_\perp(\mathbb{R}^{3};\mathbb{C}^{3}))=\bigoplus_{n=0}^{+\infty}\bigvee^{n}L^{2}_\perp (\mathbb{R}^{3};\mathbb{C}^{3})\,,
\end{equation*}
where $L^{2}_\perp(\mathbb{R}^{3};\mathbb{C}^{3}) = \{\vec{f} \in L^2(\mathbb R^3; \mathbb C^3) \mid \forall k \in \mathbb R^3, k\cdot \vec f(k) = 0 \}$.
The free field Hamiltonian in momentum representation is the second
quantization of the multiplication operator by the euclidean norm
of~$k$,
\begin{equation*}
\mathbb{H}_{\text{f }}:=\mathrm{d}\Gamma(|k|)\,.
\end{equation*}

The kinetic energy of the electron minimally coupled to the field is given by the following expression, which is quadratic in the creation and annihilation operators, 
\begin{equation*}%\label{eq:kinet_en}
\left(\vec{\sigma} \cdot \big(-i\vec{\nabla}_{x}\otimes \mathbf{I}_{\mathbb{C}^2} \otimes\mathbf{I}_{\text{f}}-\vec{\mathbb{A}}(\vec{m}_{x})\big)\right)^{2},
\end{equation*}
where  
\[
\vec{\mathbb{A}}(\vec{m}_{x}):=(a(\vec{m}_{x,j})+a^{*}(\vec{m}_{x,j}))_{1\leq j\leq 3}
\]
has three components, corresponding to the three components for $1\leq j\leq 3$ of the coupling functions
\[
\vec{m}_{x,j}(k,\tau):=g\frac{\chi(k)}{|k|^{1/2}}e^{-ik\cdot x}\vec{\varepsilon}_{\tau,j}(k)\,,
\]
and the Pauli matrices are
\[
\sigma_{1}=\begin{pmatrix}0 & 1\\
1 & 0
\end{pmatrix}\,,\;\sigma_{2}=\begin{pmatrix}0 & -i\\
i & 0
\end{pmatrix}\,,\;\sigma_{3}=\begin{pmatrix}1 & 0\\
0 & -1
\end{pmatrix}\,.
\]
The coupling functions are defined using a family $(\vec{\varepsilon}_{\tau}(k))_{\tau \in \{1,2,3\}}$
of polarization vectors, i.e. orthonormal bases of $\mathbb{R}^{3}$ depending on~$k$ in~$\mathbb{R}^{3}\setminus\{\vec{0}\}$ and
such that $\vec{\varepsilon}_{3}(k)={k}/{|k|}$, a coupling constant~$g$ in~$\mathbb{R}$
and an ultraviolet cutoff function $\chi$ such that~$\chi/|k|^{1/2}$ and~$\chi/|k|$ are both in~$L^2(\mathbb{R}^3)$. Note, though, that these conditions on~$\chi$ will be relaxed to some extent in our study of the Maxwell--Schr\"odinger functional.
 
The Pauli-Fierz Hamiltonian of the standard model of non-relativistic QED is given by
\begin{align}
\mathbb{H}
:= & \left(\vec{\sigma} \cdot \big(-i\vec{\nabla}_{x}\otimes \mathbf{I}_{\mathbb{C}^2} \otimes\mathbf{I}_{\text{f}}-\vec{\mathbb{A}}(\vec{m}_{x})\big)\right)^{2}
+V \otimes \mathbf{I}_{\mathbb{C}^2} \otimes\mathbf{I}_{\text{f}}+\mathbf{I}_{\text{el}}\otimes \mathbb{H}_{\text{f}} \notag \\
= & \big(-i\vec{\nabla}_{x}\otimes \mathbf{I}_{\mathbb{C}^2} \otimes\mathbf{I}_{\text{f}}-\vec{\mathbb{A}}(\vec{m}_{x})\big)^{2}
-\vec{\sigma} \cdot \sqrt{2}\Phi(\vec{\nabla}_x \wedge \vec{m}_x)
+V \otimes \mathbf{I}_{\mathbb{C}^2} \otimes\mathbf{I}_{\text{f}}+\mathbf{I}_{\text{el}}\otimes \mathbb{H}_{\text{f}} \, , \label{eq:defH}
\end{align}
where the normalization of the field operator is given in the Appendix, see~\eqref{eqn:def-field-operator}. 
The operator $\mathbb{H}$ on $\mathcal{H}_{\text{el}}\otimes\mathcal{H}_{\text{f}}=L^2(\mathbb{R}^3;\mathbb{C})\otimes\mathbb{C}^2\otimes\mathcal{H}_{\text{f}}$ identifies with a self-adjoint operator with form domain 
\begin{equation}\label{eq:def-Hfree-PF}
\mathcal{Q}(\mathbb{H}) := \mathcal{Q}(\mathbb{H}_{\mathrm{free}}), \quad \mathbb{H}_{\mathrm{free}}:=H_V\otimes \mathbf{I}_{\mathbb{C}^2} \otimes\mathbf{I}_{\text{f}}+\mathbf{I}_{\text{el}}\otimes \mathbb{H}_{\text{f}} \,,
\end{equation}
see Appendix \ref{app:Fock}, where $H_V$ is defined in \eqref{eq:HV}. Under suitable assumptions on $V$ and $\chi$, one can actually check that $\mathcal{D}(\mathbb{H}) := \mathcal{D}(\mathbb{H}_{\mathrm{free}})$, see \cite{Hiroshima02,HaslerHerbst08}.

\subsection{The Maxwell--Schr\"odinger energy functional}\label{subsec:MS}

We take $u$ in~$\mathcal{U}$ 
and consider a coherent state
\[
\Psi_{\vec{f}}=e^{i\Phi(\frac{\sqrt{2}}{i}\vec{f})}\Omega
\in\mathcal{H}_{\text{f}}
\]
with parameter $\vec{f}$ in $L^{2}_\perp(\mathbb{R}^{3};\mathbb{C}^{3})\cap\ZZ$. Here
\begin{equation}\label{eq:defZ_PF}
\ZZ :=\left\{\vec{f}(k)=\sum_{1\leq\tau\leq 2} f_\tau(k)\vec{\varepsilon}_\tau(k)
 \mid \;k\mapsto\abs{k}^{1/2}\vec{f}(k)\in L^2(\R^3,\diff k)\right\}\,.
\end{equation}
 A direct
computation (see Section \ref{sec:MaxwellSchroedinger}) yields the following  formula for the energy of the product state $u\otimes \Psi_{\vec{f}}$ assuming that $\chi(-k) = \overline{\chi(k)}$:
\begin{equation}\label{eq:defE_PF}
\big\langle (u\otimes \Psi_{\vec{f}}),\mathbb{H} (u\otimes \Psi_{\vec{f}})\big\rangle_{\mathcal{H}}
 = 2g^2 \| |k|^{-1/2} \chi(k) \|^2_{L^{2}}
 + \big\langle \vec{f}_- , |k|\vec{f}_-\big\rangle_{L^2}
 +\mathcal{E}_V(u,\vec{A}_{\vec{f}}),
\end{equation}
where $\mathcal{E}_V$ is defined by \eqref{eq:comput_E(u,A_intro},
\begin{equation}\label{eq:def_Af}
\vec{A}_{\vec{f}}:=2 \mathcal{F}(\overline{\vec{f}_+(k)} |k|^{-1/2}) \,,%
\end{equation}
and we have set $\vec{f}_+(k) := \frac12( \vec{f}(k)+\overline{\vec{f}(-k)})$, $\vec{f}_-(k) := \frac12( \vec{f}(k)-\overline{\vec{f}(-k)})$. Note that
\begin{align}\label{eq:comput_E(u,A)}
\mathcal{E}_V(u,\vec{A}) 
%:=& \| \vec{\sigma} \cdot (-i\vec{\nabla} - g\hat\chi * \vec{A}) \, u \|^2_{L^{2}} + \langle u, V u\rangle_{L^2}
%+\frac{1}{32\pi^3}\| \vec A \|^2_{\dot{H}^1}\,, \\
=&
\| (-i\vec{\nabla} - g\hat\chi * \vec{A}) \, u \|^2_{L^{2}} + \langle u, (V - g\hat{\chi}*\vec{\sigma} \cdot \vec{B}) u\rangle_{L^2}
+\frac{1}{32\pi^3}\| \vec A \|^2_{\dot{H}^1},
\end{align}
with $\vec{B}=\vec{\nabla}\wedge \vec{A}$\,.
We thus obtain the stationary Maxwell--Schr\"odinger energy functional in the Coulomb gauge introduced in \cite{NakamitsuTsutsumi86}, which we refer to as the Maxwell--Schr\"odinger energy functional.

If $V(x)=V(-x)$, this energy functional is invariant under Kramer's symmetry,
\begin{equation}\label{eq:Kramers}
\mathcal{E}_V(\nu u,\vec{A}(-\cdot)) =\mathcal{E}_V(u,\vec{A}),
\end{equation}
 where~$\nu u \, (x) = \sigma_2 \overline{u(-x)}$, see~\cite{LossMyaoSpohn09}.
Hence, in general, we can only hope for uniqueness of the minimizer modulo this symmetry.

As we will see in Section \ref{sec:Pauli-Fierz}, the Maxwell--Schr\"odinger energy functional is well-defined when~$(u,\vec A)$ belongs to $\mathcal{U} \times \mathcal{A}$, where
\begin{align}
\mathcal{A} := \{ \vec A\in \dot{H}^1(\mathbb{R}^3;\mathbb{R}^3) \mid  \vec \nabla \cdot \vec A = 0\}  \label{eq:def_space_A}
\end{align}
and $\chi$ satisfies the following assumption:
\begin{assumption}[Conditions on $\chi$]\label{condChi}
The cutoff function $\chi:\mathbb{R}^3\to\mathbb{R}$ satisfies $\chi(-k) = \chi(k)$ for all~$k$ in~$\mathbb{R}^3$ and 
\begin{equation*}
\frac{\chi}{|k|}\in L^2(\mathbb{R}^3)+L^{3,\infty}(\mathbb{R}^3)\,.
\end{equation*}
\end{assumption}
\begin{remark}
At the expense of slightly more involved expressions in some places, our main results below hold under the more general assumption that $\chi$ is complex-valued and satisfies $\chi(-k)=\overline{\chi(k)}$ for all $k$ in $\mathbb{R}^3$.
\end{remark}
The main quantity studied in this paper is
\[
E_{U}:=\inf_{\mathcal{U\times A}}\mathcal{E}_{U} \,,
\]
with~$U$ a potential satisfying Hypothesis~\ref{condVPauliFierz}.

\subsection{Main results}\label{subsec:main}

We begin with the following proposition which relates minimizers of the Maxwell--Schr\"odinger energy functional to minimizers of the energy of product states $u\otimes\Psi_{\vec{f}}$ in the standard model of non-relativistic QED.
\begin{proposition}\label{prop:equiv_PF}
Suppose that $V$ satisfies Hypothesis~\ref{condVPauliFierz} and that $\chi$ satisfies Hypothesis~\ref{condChi}. If~$(u_{\mathrm{gs}}, \vec{A}_{\mathrm{gs}})$ is a global minimizer of~$\mathcal{E}_V$ over~$\mathcal{U\times A}$, then there exists $\vec{f}_{\mathrm{gs}}$ in~$L^{2}_\perp(\mathbb{R}^{3};\mathbb{C}^{3})\cap\ZZ$ such that~$\vec{A}_{\mathrm{gs}}=\vec{A}_{\vec{f}_\mathrm{gs}}$ in the sense of \eqref{eq:def_Af}.
\end{proposition}
This result shows that, up to the trivial renormalization consisting in removing the $\chi$-dependent constant obtained from normal-ordering the Hamiltonian $\mathbb{H}$, the minimizers of the energy of product states $u\otimes \Psi_{\vec{f}}$ in the standard model of non-relativistic QED can be computed \emph{via} the Maxwell--Schr\"odinger energy functional. More precisely,
\begin{equation}\label{eq:defE_PF2}
\min_{(u,\vec{f})\in\mathcal{U}\times (L^2_\perp\cap\ZZ)}  \big\langle (u\otimes \Psi_{\vec{f}}),\mathbb{H} (u\otimes \Psi_{\vec{f}})\big\rangle - 2g^2 \Big\| \frac{\chi(k)}{\sqrt{|k|}} \Big\|^2_{L^{2}} = \min_{(u,\vec{A})\in\mathcal{U\times A}}\mathcal{E}_{V}(u,\vec{A}),
\end{equation}
the minimizers in both sides of the equality (if they exist) being related as in \eqref{eq:def_Af}.

Our main result concerning the existence of a minimizer for $\mathcal{E}_V$ is the following.
\begin{theorem}[Existence of a ground state for Maxwell--Schr\"odinger]\label{thm:Pauli-Fierz}
Suppose that $V$ satisfies Hypothesis~\ref{condVPauliFierz} and that $\chi=\chi_1+\chi_2$ satisfies Hypothesis~\ref{condChi} with~$\chi_1/|k|$ in~$L^2$ and $\chi_2/|k|$ in~$L^{3,\infty}$. Suppose that the decomposition $V=V_1+V_2$ of Hypothesis~\ref{condVPauliFierz} can be chosen such that~$E_{V_1}>E_V$.
With the constant~$a\ge0$ from Hypothesis~\ref{condVPauliFierz} and some universal constant~$C>0$ (see Lemma~\ref{lem:estimateChiAU}), if
\begin{equation}\label{eq:smallness_condition}
32\pi^3 a C^2 g^2 \Big\|\frac{\chi_2}{|k|}\Big\|_{L^{3,\infty}}^2<1 \,,
\end{equation}
then the Maxwell--Schr\"odinger energy functional~$\mathcal{E}_V$
admits a minimizer $(u_\mathrm{gs},\vec{A}_\mathrm{gs})$ in~$\mathcal{U\times A}$. 
\end{theorem}

\begin{remark}
For $|g|\, \|{\chi_2}/{|k|}\|_{L^{3,\infty}}$ sufficiently small, the existence of a ground state holds without assuming the presence of an ultraviolet cutoff. The case $\chi=1$ is indeed covered by the previous theorem, since~$1/|k|$ belongs to~$L^{3,\infty}$.
\end{remark}

\begin{remark}
The smallness condition \eqref{eq:smallness_condition} only concerns the critical part $\chi_2$ such that~$\chi_2/|k|$ belongs to $L^{3,\infty}$. We do not require any restriction on $\|\chi_1/|k|\|_{L^2}$.
\end{remark}

\begin{remark}\label{rk:cond_gap}The condition $E_{V_1}>E_V$ is verified in many cases of interest:
\begin{itemize}
\item For potentials $V$ such that~$\mu_V<0$ (e.g. if $V$ is a negative Coulomb potential), one has $E_V\leq \mu_V < 0\leq E_{V_1}$. 
\item For confining potentials (i.e. such that $V(x)\to \infty$ as $|x|\to \infty$), Lemma~\ref{lm:confining_intro} and Proposition~\ref{lem:delta-E-geq-delta-e} imply that there always exists a decomposition $V=V_1+V_2$ such that~$E_V<E_{V_1}$.
\item Assuming the `binding' condition $\mu_{V_1}>\mu_V$ and that $|g| \big\|\frac{\chi}{|k|}\big\|_{L^2+L^{\infty,3}}\le C_V$ with $C_V$ small enough, Proposition~\ref{lem:delta-E-geq-delta-e} implies that~$E_{V_1}>E_V$.
\end{itemize}
\end{remark}

\begin{remark}
If one considers a spinless particle instead of an electron, then the previous theorem becomes trivial. Indeed, using the diamagnetic inequality, it is not difficult to verify that the Maxwell--Schr\"odinger energy of a spinless particle reaches its minimum when~$\vec{A}=0$. On the contrary, Proposition \ref{prop:Asymtotic-Expansion-Ground-State-Energy} below shows that the minimizer of the Maxwell--Schr\"odinger energy for a spin-$\frac12$ electron is not trivial in general.
\end{remark}

\begin{remark}
If $(u,\vec{A})$ is a minimizer of the Maxwell-Schr\"odinger energy functional~$\mathcal{E}_V$, then, by Kramer's symmetry \eqref{eq:Kramers}, $(\nu u,\vec{A}(-\cdot))$ is another minimizer of~$\mathcal{E}_V$, different from the first one since $\nu u\perp u$.  We conjecture that, for $g>0$ sufficiently small, there are exactly two minimizers for~$\mathcal{E}_V$, up to the phase symmetry with respect to~$u$.
\end{remark}

To prove Theorem \ref{thm:Pauli-Fierz} we apply the usual strategy from the calculus of variations \cite{LionsI84,LionsII84}, considering a minimizing sequence $(u_j,\vec{A}_j)$ in $\mathcal{U}\times\mathcal{A}$ and proving that it converges, along some subsequence, to a minimizer of $\mathcal{E}_V$. A difficulty here comes from the fact that the minimization problem is subject to a constraint on the parameter $u$, but not on $\vec{A}$. We first establish a suitable coercivity property that allows us to localize possible minimizers to a ball in $\mathcal{U}\times\mathcal{A}$. This implies that $(u_j,\vec{A}_j)$ converges weakly to some $(u_\infty,\vec{A}_\infty)$ in $\mathcal{U}\times\mathcal{A}$. Then we can the relative compactness in $L^2$ of a ball in $H^1$ to deduce that $(u_j)$ converges strongly in~$L^2$ to $u_\infty$. This in turn suffices to prove the existence of a minimizer.

The main difficulty to implement this approach comes from the presence of singular terms in the interaction (i.e.~terms involving~$\chi_2$ with~$\chi_2/|k|$ in~$L^{3,\infty}$). In order to handle them, we use suitable estimates in Lorentz spaces that we detail in the next section. This is one of the main novelties of this paper, which allows us to remove the ultraviolet cutoff, and which we believe is naturally suited to study the minimization problem in the present context.

Our next proposition establishes the asymptotic expansion of the ground state energy $E_V$ up to third order in the coupling constant, assuming that $V$ and $\chi$ are radial. 
\begin{proposition}[Asymptotic expansion of the ground state energy at small coupling]\label{prop:Asymtotic-Expansion-Ground-State-Energy}
Suppose that $V$ satisfies Hypothesis~\ref{condVPauliFierz} and~\ref{condGS}, and $\chi$ satisfies Hypothesis~\ref{condChi}. 
Suppose also that $V$ and $\chi$ are radial, and that the decomposition $V=V_1+V_2$ of Hypothesis~\ref{condVPauliFierz} can be chosen such that $E_{V_1}>E_V$.
There exist $\varepsilon_V>0$ and $C_V>0$ such that, if
\begin{equation*}
\mathbf{g}_\chi:=|g|\Big\|\frac{\chi}{|k|}\Big\|_{L^2+L^{3,\infty}}\le\varepsilon_V \,,
\end{equation*}
then the minimum of the energy satisfies
\begin{align}\label{expansion-pauli-fierz}
\left| E_V - \mu_V - \frac{32}{3}\pi^3 \int (g\hat\chi * u_V^2)^2  \right| \leq C_V \, \mathbf{g}_\chi^4 \,.
\end{align}
In particular, if $\chi = 1$, then
\begin{align*}
\left| E_V - \mu_V - g^2\frac43 (8\pi^3)^3  \int u_V^4 \right| \leq C_V \, \mathbf{g}_\chi^4 \,.
\end{align*}
\end{proposition}
\begin{remark}
The asymptotic expansion at small coupling of the ground state energy of the Hamiltonian $\mathbb{H}$ in the standard model of non-relativistic QED has been computed in \cite{BachFrohlichPizzo09}.
\end{remark}
To prove Proposition \ref{prop:Asymtotic-Expansion-Ground-State-Energy}, we derive Euler-Lagrange type equations for minimizers $(u_{\mathrm{gs}},\vec{A}_{\mathrm{gs}})$, which we subsequently project to the vector space spanned by the electronic ground states and its orthogonal complement. The asymptotic expansion in Proposition~\ref{prop:Asymtotic-Expansion-Ground-State-Energy} then follows from estimating these equations.

Our last concern is to prove the convergence of the ground state energies in the ultraviolet limit. More precisely, suppose that the interaction between the electron and the field is cut-off in the ultraviolet, i.e.~that the Maxwell--Schr\"odinger energy functional is given by
\begin{equation}\label{eq:comput_E(u,A)Lambda_intro}
\mathcal{E}_{V,\Lambda}(u,\vec{A}) :=
\big\| \vec{\sigma}\cdot (-i\vec{\nabla} - g\hat\chi_\Lambda * \vec{A}) \, u \big\|^2_{L^{2}} + \langle u, V u\rangle_{L^2}
+\frac{1}{32\pi^3}\| \vec A \|^2_{\dot{H}^1},
\end{equation}
with $\chi_\Lambda=\chi\mathds{1}_{|k|\le\Lambda}$, for some ultraviolet parameter $\Lambda>0$. Define the ground state energies $E_{V,\Lambda}$  by
\begin{equation}\label{eq:GSenergy_Lambda}
E_{V,\Lambda}:=\inf_{(u,\vec{A})\in\mathcal{U}\times\mathcal{A}}\mathcal{E}_{V,\Lambda}(u,\vec{A}).
\end{equation}
The next proposition then shows that $E_{V,\Lambda}$ converges to $E_V$ as $\Lambda\to\infty$. 
\begin{proposition}[Ultraviolet limit of the ground state energies]\label{prop:conv-GS-en-PF_intro}
Suppose that $V$ satisfies Hypothesis \ref{condVPauliFierz} and that $\chi$ satisfies Hypothesis \ref{condChi}  and $32\pi^{3}aC^{2}g^{2}\|\chi_{2}/|k|\|_{L^{3,\infty}}^{2}<1$. Then
\begin{equation*}
E_{V,\Lambda} \underset{\Lambda\to\infty}{\longrightarrow} E_{V}.
\end{equation*}
\end{proposition}
Note that the conditions imposed in Proposition \ref{prop:conv-GS-en-PF_intro} are weaker than those ensuring the existence of a ground state in Theorem \ref{thm:Pauli-Fierz}.

\subsection{Organisation of the paper}
In the preliminary Section~\ref{sec:Prelim}, we state estimates on the electronic Hamiltonian. Most of the proofs can be found in the companion paper~\cite{BreteauxFaupinPayet22_1}. We also establish functional inequalities in Lorentz spaces used to handle the ultraviolet limit in the Maxwell-Schr\"odinger energy functional. Our main results are proved in Section~\ref{sec:Pauli-Fierz}: We first reduce the variational problem for \eqref{eq:quasi-cl}  to the minimization of the Maxwell--Schr\"odinger energy functional \eqref{eq:comput_E(u,A_intro} in Section~\ref{sec:MaxwellSchroedinger}. Existence of a minimizer for the Maxwell--Schr\"odinger energy as stated in Theorem~\ref{thm:Pauli-Fierz} is proved in Section~\ref{sec:PF-existence}. In Section~\ref{sec:PF-uniqueness}, we establish useful properties of the set of minimizers. We compute the expansion of the ground state energy for small coupling constants and prove Proposition \ref{prop:Asymtotic-Expansion-Ground-State-Energy} in Section~\ref{sec:PF-expansion}. Finally, the convergence of the ground state energies in the ultraviolet limit (Proposition \ref{prop:conv-GS-en-PF_intro}) is proved in Section~\ref{sec:UV-PF}. For the sake of completeness, the self-adjointness of the Pauli-Fierz Hamiltonian and its quadratic form domain are recalled in Appendix~\ref{app:Fock}.

\section{Preliminaries}\label{sec:Prelim}

In this preliminary section, we gather several technical estimates that will be used in the next section to prove our main results. The first subsection mainly concerns the electronic Hamiltonian $H_V$. We refer to the article \cite{BreteauxFaupinPayet22_1} for a proof of some of the stated results.
In a second subsection, we give some functional estimates in Lorentz spaces that will be used in a crucial way to control the interactions terms in the Maxwell--Schr\"odinger energy functional. 

\subsection{Estimates on the electronic part}\label{subsec:Estimate-elec}

Recall that our assumptions on the external potential $V$ of the electronic Hamiltonian $H_V=-\Delta+V$ have been introduced in Section \ref{subsec:elec}. We begin with a few remarks showing that $H_V$ is well-defined and that $\mathcal{Q}(H_V)=\mathcal{Q}(H_{V_+}) = \mathcal{Q}(H_{V_1})$ with $V_1$ as in Hypothesis \ref{condVPauliFierz}.

First, $V_-$ is form bounded with respect to $\sqrt{-\Delta}$, by Hypothesis \ref{condVPauliFierz}. 
This implies by a well-known argument that $V_-$ is also form bounded with respect to $H_{V_+}$ with a relative bound less than $1$, and hence the KLMN Theorem (see \cite[Theorem X.17]{ReedSimonII}) yields that $H_V$ identifies with a semi-bounded self-adjoint operator with form domain $\mathcal{Q}(H_V) = \mathcal{Q}(H_{V_+})$.

Next, Hypothesis \ref{condVPauliFierz}(ii) implies that $V_2$ is relatively form bounded with respect to $\sqrt{-\Delta}$ with relative bound $0$. Indeed, for $R$ sufficiently large, $V_2\mathds{1}_{|x|\ge R}$ belongs to~$L^\infty(\mathbb{R}^3)$ since~$V_2(x)\to0$ as $|x|\to\infty$, while $V_2\mathds{1}_{|x|\le R}$ belongs to~$L^{3/2}(B_R)$ with $B_R:=\{x\in\mathbb{R}^3\, | \, |x|\le R\}$, since~$V_2$ is in~$L^{3/2}_{\mathrm{loc}}(\mathbb{R}^3)$. Therefore $V_2$ belongs to~$L^{3/2}(\mathbb{R}^3) + L^\infty(\mathbb{R}^3)$ and hence we can apply \cite[Theorem X.19]{ReedSimonII} to deduce that $V_2$ is infinitesimally form-bounded with respect to $\sqrt{-\Delta}$.  In turn, since $V_+-V_1=V_2+V_-$ is form bounded with respect to $\sqrt{-\Delta}$, it is not difficult to verify that $\mathcal{Q}(H_{V_+}) = \mathcal{Q}(H_{V_1})$.

We recall a version of the IMS localization formula (see e.g. \cite{CFKS87}), used to split the contributions to the energy for large~$x$ and for small~$x$. We state it for a magnetic kinetic energy since this context is relevant in Section \ref{sec:Pauli-Fierz} to study the Maxwell--Schr\"odinger energy functional.

\begin{lemma}[Magnetic IMS localization formula]\label{lem:MagneticIMS}
Let $\vec{A}\in L_{\mathrm{loc}}^{2}(\mathbb{R}^{3};\mathbb{R}^{3})$ and $\eta,\,\tilde{\eta}:\mathbb{R}^3\to\mathbb{R}$ be differentiable with bounded first derivatives and such that $\eta^{2}+\tilde{\eta}^{2}=1$. Let
$u\in H_{\vec{A}}^{1}=\{\tilde{u}\in L^{2}\mid(-i\vec{\nabla}-\vec{A})\tilde{u}\in L^{2}\}$.
Then
\begin{equation}
\|(-i\vec{\nabla}-\vec{A}) u\|^2 
 =\|(-i\vec{\nabla}-\vec{A})\,\eta \, u\|^2+\|(-i\vec{\nabla}-\vec{A}) \, \tilde{\eta} \, u\|^2 
+\langle u,(|\vec{\nabla}\eta|^{2}+|\vec{\nabla}\tilde{\eta}|^{2})u\rangle\,.\label{eq:IMS-magnetic}
\end{equation}
\end{lemma}

\begin{proof}
Using the commutator $[-i\vec{\nabla}-\vec{A},\eta]=(-i\vec{\nabla}\eta)$
three times yields
\begin{multline}
\langle (-i\vec{\nabla}-\vec{A}) \, u,(-i\vec{\nabla}-\vec{A}) \,\eta^{2}u\rangle
=\langle (-i\vec{\nabla}-\vec{A}) \, \eta u,(-i\vec{\nabla}-\vec{A}) \,\eta u\rangle-\langle u,(-i\vec{\nabla}\eta)^{2}u\rangle\\
+\frac{1}{2}\langle u,(-(-i\vec{\nabla}\eta^{2})(-i\vec{\nabla}-\vec{A})+(-i\vec{\nabla}-\vec{A})\,(-i\vec{\nabla}\eta^{2}))u\rangle\,.\label{eq:IMS-part-chi-R}
\end{multline}
Summing (\ref{eq:IMS-part-chi-R}) and the same equation with
$\eta$ replaced by $\tilde{\eta}$ leads to
\begin{multline*}
\|(-i\vec{\nabla}-\vec{A}) u\|^2 
 =\|(-i\vec{\nabla}-\vec{A})\,\eta \, u\|^2+\|(-i\vec{\nabla}-\vec{A}) \, \tilde{\eta} \, u\|^2 
+\langle u,(|\vec{\nabla}\eta|^{2}+|\vec{\nabla}\tilde{\eta}|^{2})u\rangle\\
+\big\langle u,\big(-(-i\vec{\nabla}\frac{\eta^{2}+\tilde{\eta}^{2}}{2})(-i\vec{\nabla}-\vec{A})+(-i\vec{\nabla}-\vec{A})(-i\vec{\nabla}\frac{\eta^{2}+\tilde{\eta}^{2}}{2})\big)u\big\rangle
\end{multline*}
which implies the result since $\eta^{2}+\tilde{\eta}^{2}$
is constant.
\end{proof}

The following lemma shows that, for confining potentials~$V$, the gap $\mu_{V_1}-\mu_V$ can be made as large as we want, provided that the potential $V_1$ is suitably chosen. The proof can be found in the companion paper~\cite{BreteauxFaupinPayet22_1}.

\begin{lemma}\label{lm:confining_intro}
Suppose that $V=V_+-V_-$ is such that
\begin{enumerate}[label=(\roman*)]
\item $V_+\in L^1_{\mathrm{loc}}(\mathbb{R}^3)$, \vspace{0,1cm}
\item $V_-\in L^{3/2}_{\mathrm{loc}}(\mathbb{R}^3)$, \vspace{0,1cm}
\item $V(x)\to\infty$ as $|x|\to\infty$.
\end{enumerate}
Then, for all $C>0$, there exist a decomposition $V=V_{1,C}+V_{2,C}$ as in Hypothesis~\ref{condVPauliFierz} such that, moreover,
\begin{equation*}
\mu_{V_{1,C}}-\mu_V\ge C.
\end{equation*}
\end{lemma}

To conclude this section, we give a lemma, proved again in the companion paper~\cite{BreteauxFaupinPayet22_1}, which is useful to prove the existence of minimizers for the energy functional studied in Section \ref{sec:Pauli-Fierz}.

\begin{lemma}
\label{lem:LowerSemiContDeltaV}Suppose that $V$ satisfies Hypothesis~\ref{condVPauliFierz}. Let $(u_{j})_{j\in\mathbb{N}}$ be a bounded sequence in~$H^{1}(\mathbb{R}^{3})$ which  converges weakly to~$u_{\infty}$ in~$H^1(\mathbb{R}^3)$, and strongly in~$L^{2}(\mathbb{R}^{3})$. Then 
\begin{equation*}
\langle u_{\infty},(-\Delta+V)u_{\infty}\rangle\leq\liminf_{j\to\infty}\langle u_{j},(-\Delta+V)u_{j}\rangle\,.
\end{equation*}
\end{lemma}

\subsection{Functional inequalities in Lorentz spaces}

In the proof of our main results, we use in a crucial way some functional inequalities in Lorentz spaces that we present in this section. For $1\le p<\infty$, the Lorentz spaces $L^{p,\infty}=L^{p,\infty}(\mathbb{R}^d)$ are defined as the set of (equivalence classes of) measurable functions $f:\mathbb{R}^d\to\mathbb{C}$ such that \eqref{eq:norm_weakLp} holds.

More generally, for $1\le p<\infty$ and $1\le q\leq\infty$, the Lorentz spaces $L^{p,q}=L^{p,q}(\mathbb{R}^d)$ are defined as the set of (equivalence classes of) measurable functions $f:\mathbb{R}^d\to\mathbb{C}$ such that the quasi-norm
\begin{equation*}
\|f\|_{L^{p,q}}:= p^{1/q} \| \lambda (  \{|f|>t\})^{1/p}\,t\|_{L^q ((0,\infty),\diff t / t)}  
\end{equation*}
is finite.

For $1\le p<\infty$ and $1\le q_1\leq q_2 \le \infty$, the continuous embedding $
L^{p,q_1}\subseteq L^{p,q_2}$ holds.
Moreover $L^{p,p}$ identifies with the Lebesgue space~$L^p$. We use the following generalizations of H\"older and Young's inequality in Lorentz spaces, see \cite{ONeil63,Yap69,Lemarie-Rieusset02,BezLeeNakamuraSawano17} or \cite[1.4.19]{Grafakos08}. 

For $1\leq p_1,p_2<\infty$, $1\le q_1,q_2\le\infty$, H\"older's inequality states that
\begin{equation}\label{eq:weak_Holder}
\|f_1f_2\|_{L^{p,q}}\lesssim\|f_1\|_{L^{p_1,q_1}}\|f_2\|_{L^{p_2,q_2}},\qquad \frac{1}{p}=\frac{1}{p_1}+\frac{1}{p_2}\,, \quad \frac{1}{q}=\frac{1}{q_1}+\frac{1}{q_2}\,,
\end{equation}
whenever the right hand side is finite.

Young's inequality states that, for $1<p,p_1,p_2<\infty$, $1\le q_1,q_2\le\infty$, 
\begin{equation}\label{eq:weak_Young}
\|f_1*f_2\|_{L^{p,q}}\lesssim\|f_1\|_{L^{p_1,q_1}}\|f_2\|_{L^{p_2,q_2}}\,, \qquad 1+\frac{1}{p}=\frac{1}{p_1}+\frac{1}{p_2}\,,\quad \frac{1}{q}=\frac{1}{q_1}+\frac{1}{q_2}\,.
\end{equation}

\subsubsection{Functional inequalities in the Maxwell--Schr\"odinger setting} 

We present estimates which will play an important role in the next section. We work in the setting of the Maxwell--Schr\"odinger energy functional introduced in \eqref{subsec:MS}, with $u$ in $\mathcal{Q}_V\subset L^2(\mathbb{R}^3;\mathbb{C}^2)$ (see \eqref{eq:QV_intro}) and $A$ in $\dot{H}^1(\mathbb{R}^3;\mathbb{R}^3)$.

\begin{lemma}\label{lem:estimateChiAU}Under Hypothesis~\ref{condChi} on~$\chi=\chi_1+\chi_2$ with~$\chi_1/|k|$ in~$L^2$ and~$\chi_2/|k|$ in~$L^{3,\infty}$, there exists a universal constant $C>0$ such that,
\begin{multline}
\forall (u,\vec A) \in H^1\times \dot{H}^1, \;
\|(\hat\chi * \vec A)u\|_{L^2}
\leq C  \|\vec A\|_{\dot{H}^1}\left(\left\|\frac{\chi_1}{|k|}\right\|_{L^2} \|u\|_{L^2}+\left\|\frac{\chi_2}{|k|}\right\|_{L^{3,\infty}}\|u\|_{\dot{H}^{1/2}} \right)  \,, \label{eq:25}
\end{multline}
and
\begin{equation}\label{eq:26}
\forall (u,\vec A) \in \mathcal{Q}_V \times \dot{H}^1, \quad
\|(\hat\chi * \vec A)u\|_{L^2}
 \leq C \left\|\frac{\chi}{|k|}\right\|_{L^2+L^{3,\infty}} \|\vec A\|_{\dot{H}^1} \|u\|_{L^2}^{1/2} \|u\|_{\mathcal{Q}_V}^{1/2} \,.
\end{equation}
\end{lemma}
\begin{proof}In this proof $X\lesssim Y$ means that there is a universal constant $c$ such that $X\leq c Y$. 
H\"older and Young's inequalities are sufficient to estimate 
\begin{multline*}
\|(\hat \chi_{1}*\vec{A})u\|_{L^{2}}
\lesssim\|(\chi_{1}\,\overline{\mathcal{F}}\vec{A})*\mathcal{\overline{\mathcal{F}}}u\|_{L^{2}}
\leq\|\chi_{1}\,\mathcal{\overline{\mathcal{F}}}\vec{A}\|_{L^{1}} \, \|u\|_{L^{2}}\\
\leq\|\chi_{1}/|k|\|_{L^{2}} \, \||k|\,\overline{\mathcal{F}}\vec{A}\|_{L^{2}} \, \|u\|_{L^{2}}\leq\|\chi_{1}/|k|\|_{L^{2}} \, \|\vec{A}\|_{\dot{H}^{1}} \,  \|u\|_{L^{2}}.
\end{multline*}
Similarly,
\[
\|(\hat \chi_{2}*\vec{A})u\|_{L^{2}}\lesssim \Big\|(\frac{\chi_{2}}{|k|}|k|\overline{\mathcal{F}}\vec{A})*(\frac{1}{|k|^{1/2}}(\overline{\mathcal{F}}u)^{1/2}|k|^{1/2}(\overline{\mathcal{F}}u)^{1/2})\Big\|_{L^{2}},
\]
which can be estimated using the H\"older and Young's inequalities in Lorentz spaces, see \eqref{eq:weak_Holder}--\eqref{eq:weak_Young} 
(or the Brascamp-Lieb inequality in Lorentz spaces, see~\cite{BennettCarberyChristTao08,BezLeeNakamuraSawano17}): for any~$\varphi$ in~$L^{2}(\mathbb{R}^{3};\mathbb{C}^2)$,
\begin{multline*}
\Big|\iint\varphi(k')\Big(\frac{\chi_{2}(k'-k)}{|k'-k|}|k'-k|\overline{\mathcal{F}}\vec{A}(k'-k)\Big) \Big(\frac{1}{|k|^{1/2}}|k|^{1/2}\overline{\mathcal{F}}u \Big)\,\diff k \diff k' \Big|\\
\lesssim \|\varphi\|_{L^{2}} \, \Big\|\frac{\chi_{2}}{|k|}\Big\|_{L^{3,\infty}}  
\||k|\overline{\mathcal{F}}\vec{A}\|_{L^{2}}  \, \Big\|\frac{1}{|k|^{1/2}}\Big\|_{L^{6,\infty}}  
\||k|^{1/2}\overline{\mathcal{F}}u\|_{L^{2}}.
\end{multline*}
By duality, one gets the estimate:
\[
\|(\hat{\chi}_{2}*\vec{A})u\|_{L^{2}}\lesssim \Big\|\frac{\chi_{2}}{|k|}\Big\|_{L^{3,\infty}}\|\vec{A}\|_{\dot{H}^{1}} \|u\|_{\dot{H}^{1/2}}\,.
\]
This proves \eqref{eq:25}. H\"older's inequality then yields \eqref{eq:26}, since
\[
\|u\|_{\dot{H}^{1/2}} 
\lesssim \left\| | \overline{\mathcal{F}}u|^{1/2}|k|^{1/2} |\overline{\mathcal{F}}u|^{1/2} \right\|_{L^2} 
\lesssim  \left\| |\overline{\mathcal{F}}u|^{1/2} \right\|_{L^{4}} \, \left\||k|^{1/2} |\overline{\mathcal{F}}u|^{1/2} \right\|_{L^{4}} 
\lesssim \|u\|_{H^{1}}^{1/2} \, \|u\|^{1/2}_{L^{2}}.
\]
This ends the proof.
\end{proof}

\begin{lemma}\label{lem:estimateChikFu1u2}
Under Hypothesis~\ref{condChi} on~$\chi$, there exists a universal constant $C>0$ such that, for all tempered distribution $w$ such that $\mathcal{F}w$ is in $L^\infty \cap L^{6,2}$,
\begin{align}\label{eq:27}
\|\hat\chi * w\|_{\dot{H}^{-1}}
\leq C \Big\|\frac{\chi}{|k|}\Big\|_{L^2+L^{3,\infty}} \|\mathcal{F}w\|_{L^\infty \cap L^{6,2}} \,,
\end{align}
and for all $u_1$ in $L^2$, $u_2$ in $H^1$, and $0<\Lambda\leq \infty$,
\begin{align}\label{eq:28}
\|\mathds{1}_{|k|\leq \Lambda}\mathcal{F}(u_1u_2)\|_{L^\infty \cap L^{6,2}}
\leq C  \|u_1\|_{L^2} \|u_2\|_{H^{1}} \,.
\end{align}
\end{lemma}
\begin{proof}
With a decomposition $\chi=\chi_1+\chi_2$ where~$\chi_1/|k|$ in~$L^2$ and~$\chi_2/|k|$ in~$L^{3,\infty}$, H\"older's inequality gives
\begin{align*}
\|\hat\chi_1 * w\|_{\dot{H}^{-1}} 
\lesssim \Big\|\frac{\chi_1}{|k|}\Big\|_{L^2} \|\mathcal{F}w\|_{L^{\infty}}   \,.
\end{align*}
Likewise, H\"older's inequality in Lorentz spaces, see~\eqref{eq:weak_Holder}, yields
\begin{align*}
\|\hat\chi_2 * w\|_{\dot{H}^{-1}} 
\lesssim \Big\|{\frac{\chi_2}{\abs{k}}\, \FF w }\Big\|_{L^2} 
\leq \Big\|{\frac{\chi_2}{\abs{k}}}\Big\|_{L^{3,\infty}}\norm{\FF w}_{L^{6,2}} \,.% \label{eq:225}
\end{align*}
This proves \eqref{eq:27}.

Now, by the continuity of the Fourier transform from $L^1$ to $L^\infty$ and H\"older's inequality,
\begin{align*}
\|\mathds{1}_{|k|\leq \Lambda} \mathcal{F}(u_1u_2)\|_{L^\infty}
\leq 
\|u_1u_2\|_{L^1} 
\lesssim \|u_1\|_{L^2} \|u_2\|_{L^{2}}   \,.
\end{align*}
Then Young's inequality in Lorentz spaces \eqref{eq:weak_Young} yields
\begin{align*}
\norm{\mathds{1}_{|k|\leq \Lambda} \FF(u_1 u_2)}_{L^{6,2}} =
\norm{\FF(u_1)*\FF(u_2)}_{L^{6,2}}
\lesssim \norm{\FF(u_1)}_{L^{2,2}}\norm{\FF(u_2)}_{L^{3/2,\infty}}\,.
\end{align*}
Using the fact that $L^{2,2}=L^2$ and H\"older's inequality in Lorentz spaces~\eqref{eq:weak_Holder} yields
\begin{align*}
 \norm{\FF(u_1)}_{L^{2,2}}\norm{\FF(u_2)}_{L^{3/2,\infty}}
\lesssim \norm{u_1}_{L^2} \, \big\|\abs{k}^{-1/2}\big\|_{L^{6,\infty}}  \, \big\|\abs{k}^{1/2}\FF(u_2)\big\|_{L^{2,\infty}} \,.% \label{eq:227}
\end{align*}
Since $L^{2,2}=L^2$ is  continuously embedded in $L^{2,\infty}$, and since $\|u_2\|_{\dot{H}^{1/2}}\leq \|u_2\|_{H^{1}}$ this yields \eqref{eq:28}.
\end{proof}

Recall that $\mathcal{Q}_V^*$ stands for the topological dual of $\mathcal{Q}_V$ (see \eqref{eq:QV_intro}) and that the space $\mathcal{A}$ has been defined in \eqref{eq:def_space_A}.

\begin{lemma}\label{lm:new}
Suppose that $V$ satisfies Hypothesis~\ref{condVPauliFierz} and that $\chi$ satisfies Hypothesis~\ref{condChi}. There exists a universal constant $C>0$ such that, for all $\vec{A}$ in $\mathcal{A}$ and $u$ in $H^1$, 
\begin{align*}
&\|(-i\vec{\nabla} u)\cdot(\hat{\chi}*\vec{A})\|_{\mathcal{Q}_V^*}\leq C \left\|\frac{\chi}{\abs{k}}\right\|_{L^2+L^{3,\infty}}\|\vec{A}\|_{\dot{H}^1}\|u\|_{H^1}^{1/2}\|u\|_{L^2}^{1/2},\\
&\| \hat{\chi}*\vec{\sigma} \cdot (\vec{\nabla}\wedge\vec{A})u\|_{\mathcal{Q}_V^*} \leq C \left\|\frac{\chi}{\abs{k}}\right\|_{L^2+L^{3,\infty}}\|\vec{A}\|_{\dot{H}^1}\|u\|_{H^1}^{1/2}\|u\|_{L^2}^{1/2}, \\
&\|(\hat{\chi}*\vec{A})^2 u\|_{\mathcal{Q}_V^*}  \le C  \left\|\frac{\chi}{\abs{k}}\right\|_{L^2+L^{3,\infty}}^2 \|\vec{A}\|^2_{\dot{H}^1} \|u\|_{H^1}^{1/2}\|u\|_{L^2}^{1/2}.
\end{align*}
\end{lemma}
\begin{proof}
By duality, using that $\vec{\nabla}\cdot\vec{A}=0$, we have
\begin{align*}
\|(-i\vec{\nabla} u)\cdot(\hat{\chi}*\vec{A})\|_{\mathcal{Q}_V^*} &= \underset{\|v\|_{\mathcal{Q}_V}=1}{\sup}\abs{\int \overline{v(x)}[(-i\vec{\nabla} u(x))\cdot(\hat{\chi}*\vec{A})(x)]\diff x} \\
& = \underset{\|v\|_{\mathcal{Q}_V}=1}{\sup}\abs{\int \overline{-i\vec{\nabla} v(x)}[ u(x)\cdot(\hat{\chi}*\vec{A})(x)]\diff x}\\
& \leq \underset{\|v\|_{\mathcal{Q}_V}=1}{\sup}\|\vec{\nabla} v\|_{L^2}\|(\hat{\chi}*\vec{A})u\|_{L^2}\\
&\leq \underset{\|v\|_{\mathcal{Q}_V}=1}{\sup}\|v\|_{\mathcal{Q}_V}\|(\hat{\chi}*\vec{A})u\|_{L^2}\leq \|(\hat{\chi}*\vec{A})u\|_{L^2} \,.
\end{align*}
This last quantity is estimated thanks to Lemma \ref{lem:estimateChiAU}:
\begin{align*}
\|(\hat{\chi}*\vec{A})u\|_{L^2} & \lesssim \left\|\frac{\chi}{\abs{k}}\right\|_{L^2+L^{3,\infty}}\|\vec{A}\|_{\dot{H}^1} \|u\|_{H^1}^{1/2}\|u\|_{L^2}^{1/2}. %\lesssim \|\vec{A}\|_{\dot{H}^1}\|u\|_{\mathcal{Q}_V} \,,
\end{align*}

The estimate of $\| \hat{\chi}*\vec{\sigma} \cdot (\vec{\nabla}\wedge\vec{A})u\|_{\mathcal{Q}_V^*}$ is analogous, using that
\begin{equation*}
\int\overline{v(x)} \big(\hat{\chi}*\vec{\sigma} \cdot (\vec{\nabla}\wedge\vec{A})u\big)(x)\mathrm{d}x
= - \big\langle ( \vec{\nabla} \wedge \vec{\sigma})v\,, (\hat\chi * \vec{A}) u \big\rangle_{L^2}.
\end{equation*}

Similarly, by duality, H\"{o}lder's inequality and Lemma~\ref{lem:estimateChiAU},
\begin{multline*}
\|(\hat{\chi}*\vec{A})^2 u\|_{\mathcal{Q}_V^*} = \underset{\|v\|_{\mathcal{Q}_V}=1}{\sup}\abs{\int \overline{v}(\hat{\chi}*\vec{A})^2  u }  \\
\leq  \|(\hat{\chi}*\vec{A})v\|_{L^2}  \|(\hat{\chi}*\vec{A})u\|_{L^2}
 \lesssim  \left\|\frac{\chi}{\abs{k}}\right\|_{L^2+L^{3,\infty}}^2 \|\vec{A}\|^2_{\dot{H}^1} \|u\|_{H^1}^{1/2}\|u\|_{L^2}^{1/2}.
\end{multline*}
This concludes the proof of the lemma.
\end{proof}

\section{Proofs of the main results}\label{sec:Pauli-Fierz}

In this section we prove our main results stated in Section \ref{subsec:main}.
In Section~\ref{sec:MaxwellSchroedinger}, we show that minimizing the Pauli-Fierz energy over coherent states is equivalent to minimizing the Maxwell--Schr\"odinger energy functional over its natural definition domain.
The existence of a minimizer stated in Theorem \ref{thm:Pauli-Fierz} is proved in Section~\ref{sec:PF-existence}, using coercivity and lower semicontinuity arguments. In Section~\ref{sec:PF-uniqueness} we study the set of minimizers of the Maxwell--Schr\"odinger energy functional for small coupling constants.  In particular, the Euler-Lagrange equations leads to useful estimates for the minimizers, which in turn allows us to obtain in Section~\ref{sec:PF-expansion} the second-order asymptotic expansions of the ground state energy at small coupling stated in Proposition \ref{prop:Asymtotic-Expansion-Ground-State-Energy}.
Finally, we prove the convergence of the quasi-classical ground state energy  in the ultraviolet limit (Proposition \ref{prop:conv-GS-en-PF_intro}) in Section~\ref{sec:UV-PF}.

\subsection{Reduction to the Maxwell--Schr\"odinger energy functional\label{sec:MaxwellSchroedinger}}

We first justify the derivation of the Maxwell--Schr\"odinger energy functional appearing in \eqref{eq:defE_PF}. To this end we compute the energy of product state $u\otimes \Psi_{\vec{f}}$ in the standard model of non-relativistic QED. Recall that~$u$ is in~$\mathcal{U}$ (see \eqref{eq:defU_intro}) and $\vec{f}$ is in~$L^{2}_\perp(\mathbb{R}^{3};\mathbb{C}^{3})\cap\mathcal{Z}$ (see \eqref{eq:defZ_PF}). 

We introduce a direct sum decomposition $\ZZ = \ZZ^+\oplus \ZZ^-$ where
\begin{equation*}
\ZZ^+:=\left\{\vec{f}\in\ZZ,\;\vec{f}(-k) = \overline{\vec{f}(k)}\quad\forall k\in\R^3 \right\},\quad\ZZ^-:=\left\{\vec{f}\in\ZZ,\;\vec{f}(-k) = -\overline{\vec{f}(k)} \quad\forall k\in\R^3\right\}.
\end{equation*}
Note that any $\vec{f}$ in~$\ZZ$ decomposes as
$\vec{f} = \vec{f}_++\vec{f}_-$
with
\begin{equation*}
\vec{f}_+(k) :=  \frac{\vec{f}(k)+\overline{\vec{f}(-k)}}{2} \in \ZZ^+,\quad \vec{f}_-(k) :=  \frac{\vec{f}(k)-\overline{\vec{f}(-k)}}{2}\in\ZZ^-.
\end{equation*}
We suppose here that $\chi/\sqrt{|k|}$ and $\chi/|k|$ belong to~$L^2(\mathbb{R}^3)$ in order for the Hamiltonian~$\mathbb{H}$ to be well-defined (see Proposition \ref{prop:self-adj}).  These assumptions will however subsequently be relaxed in our study of the Maxwell--Schr\"odinger engergy functional.
In this section we drop the index~$V$ for~$\mathcal{E}_V$ as the potential remains fixed throughout this section.

\begin{proposition}
Let $\chi:\mathbb{R}^3\to\mathbb{R}$ be such that $\chi(-k)=\chi(k)$ for all~$k$ in~$\mathbb{R}^3$ and both~$\chi/\sqrt{|k|}$ and $\chi/|k|$ belong to~$L^2(\mathbb{R}^3)$. Let~$u$ in~$\mathcal{U}$ and~$\Psi_{\vec{f}}$ in~$\mathcal{H}_{\mathrm{f}}$ be a coherent state of parameter~$\vec{f}$ in~$L^{2}_\perp(\mathbb{R}^{3};\mathbb{C}^{3})\cap\ZZ.$ 
The energy of the state $u\otimes \Psi_{\vec{f}}$ satisfies
\begin{equation}\label{eq:comput_energy}
\Big\langle( u\otimes \Psi_{\vec{f}})\,, \,\mathbb{H} (u\otimes \Psi_{\vec{f}})\Big\rangle_{\mathcal{H}} = 2g^2\Big\| \frac{\chi(k)}{\sqrt{|k|}}\Big\|^2_{L^2}+\ps{\vec{f}_-}{|k|\vec{f}_-}_{L^2}+\E(u,\vec{A}_{\vec{f}}),
\end{equation}
where $\vec{A}_{\vec{f}}$ is given by \eqref{eq:def_Af} and ${\E}(u,\vec{A}_{\vec{f}})$ is given by \eqref{eq:comput_E(u,A)}. Moreover,
\begin{align}
&\underset{u\in\mathcal{U},\vec{f}\in L^2_\perp\cap\ZZ}{\inf}\ps{(u\otimes \Psi_{\vec{f}})\,}{\,\mathbb{H}(u\otimes \Psi_{\vec{f}})} =2g^2\Big\| \frac{\chi(k)}{\sqrt{|k|}}\Big\|^2_{L^2}+ \underset{u\in\mathcal{U},\vec{f}\in L^2_\perp\cap\ZZ^+}{\inf}\E(u,\vec{A}_{\vec{f}}). \label{eq:min_PF_Z+}
\end{align}
\end{proposition}

\begin{proof}
Using the identities \eqref{eq:coher-app} recalled in Appendix \ref{app:Fock}, we can compute the Pauli-Fierz energy of the state $u\otimes \Psi_{\vec{f}}$, which gives
\begin{multline*}
\big\langle u\otimes \Psi_{\vec{f}}\, , \,\mathbb{H}\; (u\otimes \Psi_{\vec{f}})\big\rangle = \langle u , H_Vu\rangle_{L^2} + 2g^2\norm{\vec{m}}^2_{L^2_k}-4g\Real\Big\langle u,-i\vec{\nabla} u\big\langle \vec{f},\vec{m}(x,.)\big\rangle_{L^2_k}\Big\rangle_{L^2_x}\\
+4g^2\Big\langle u, \Big( \Real\big\langle\vec{f},\vec{m}(x,.)\big\rangle_{L^2_k}\Big)^2u\Big\rangle_{L^2_x}
-g\langle u, \vec{\sigma}\cdot \vec{\nabla}_x \wedge 2\Real\langle \vec{f},\vec{m}(x,\cdot)\rangle_{L^2_k} \, u\rangle
+\big\langle \vec{f},\abs{k}\vec{f}\big\rangle_{L^2_k},
\end{multline*}
 where $\vec{m}=\sum_\tau\vec{m}_\tau$ and
\begin{equation*}
\vec{m}_\tau(x,k) := \chi(k)\abs{k}^{-1/2}e^{-ik\cdot x}\vec{\varepsilon}_\tau(k)\,.
\end{equation*}
First, we can use the properties of the Fourier transform to obtain 
\begin{align*}
\big\langle \vec{f},\vec{m}(x,.)\big\rangle_{L^2_k} &= \sum_\tau\int\overline{\vec{f}_\tau(k)}\chi(k)\abs{k}^{-1/2}e^{-ikx}\vec{\varepsilon}_\tau(k)\diff k\\
& = \FF\left(\overline{\vec{f}(k)}\abs{k}^{-1/2}\chi(k)\right)(x)\\
& = \hat{\chi}*\FF(\overline{\vec{f}(k)}\abs{k}^{-1/2}) \, (x),
\end{align*}
which means that, with the notation $\vec{A}_{\vec{f}}$ introduced in~\eqref{eq:def_Af}, and using that $\hat{\chi}$ is real valued:
\begin{equation*}
2\Real\big\langle \vec{f},\vec{m}(x,.)\big\rangle_{L^2_k} = \hat{\chi}*\vec{A}_{\vec{f}}(x) \, .
\end{equation*}
Integrating by parts then gives
\begin{align*}
2&\Real\Big\langle u,-i\vec{\nabla} u\big\langle\vec{f},\vec{m}(x,.)\big\rangle_{L^2_k}\Big\rangle_{L^2_x} \\
&= \Big\langle u,-i\vec{\nabla} u\big\langle \vec{f},\vec{m}(x,.)\big\rangle_{L^2_k}\Big\rangle_{L^2_x}+\Big\langle -i\vec{\nabla} u\big\langle \vec{f},\vec{m}(x,.)\big\rangle_{L^2_k},u\Big\rangle_{L^2_x}\\
& = \int \overline{-i\vec{\nabla} u(x)}u(x)\big\langle \vec{f},\vec{m}(x,.)\big\rangle_{L^2_k}\diff x+\int \overline{-i\vec{\nabla} u(x)}u(x)\overline{\big\langle \vec{f},\vec{m}(x,.)}\big\rangle_{L^2_k} \diff x\\
& = \int \overline{-i\vec{\nabla} u(x)}u(x)2\Real\big\langle \vec{f},\vec{m}(x,.)\big\rangle_{L^2_k}\diff x\\
& = \ps{-i\vec{\nabla} u}{(\hat{\chi}*\vec{A}_{\vec{f}})u}_{L^2}.
\end{align*}
Now, we compute the scalar product 
\begin{equation*}
\big\langle \vec{f},\abs{k}\vec{f}\big\rangle_{L^2} = \big\langle\vec{f}_+,\abs{k}\vec{f}_+\big\rangle_{L^2}+\big\langle\vec{f}_-,\abs{k}\vec{f}_-\big\rangle_{L^2}+2\Real\big\langle\vec{f}_+,\abs{k}\vec{f}_-\big\rangle_{L^2}.
\end{equation*}
Using a change of variables and the definitions of $f_+$ and $f_-$ yields
\begin{align*}
\big\langle\vec{f}_+,\abs{k}\vec{f}_-\big\rangle &= \int\overline{\vec{f}_+(k)}\abs{k}\vec{f}_-(k)\diff k = \int\overline{\vec{f}_+(-k)}\,\abs{-k}\,\vec{f}_-(-k)\diff k = -\int \vec{f}_+(k)\abs{k}\overline{\vec{f}_-(k)}\diff k.
\end{align*}
This shows that $2\Real\big\langle \vec{f}_+,\abs{k}\vec{f}_-\big\rangle = 0$. Then, 
applying the inverse Fourier transform to~\eqref{eq:def_Af} yields
\begin{equation*}
\overline{\vec{f}_+(k)}=\frac{1}{2}\abs{k}^{1/2}\FF^{-1}(\vec{A}_{\vec{f}})\,.
\end{equation*}
Finally, using Parseval's equality yields 
\begin{equation*}
\big\langle\vec{f}_+,\abs{k}\vec{f}_+\big\rangle_{L^2} = \frac{1}{32\pi^3}\big\langle\FF^{-1}(\vec{A}_{\vec{f}}),\abs{k}^{2}\FF^{-1}(\vec{A}_{\vec{f}})\big\rangle_{L^2}=\frac{1}{32\pi^3}\big\langle\vec{A}_{\vec{f}},-\Delta \vec{A}_{\vec{f}}\big\rangle_{L^2}=\frac{1}{32\pi^3}\big\|\vec{A}_{\vec{f}}\big\|^2_{\dot{H}^1}.
\end{equation*}
This allows us to obtain \eqref{eq:comput_energy}.

Now, since the term $\langle\vec{f}_-,\abs{k}\vec{f}_-\rangle$ is non-negative, we can write
\begin{align*}
\inf_{u\in\mathcal{U},\vec{f}\in L^2_\perp\cap\ZZ} &\ps{(u\otimes \Psi_{\vec{f}})}{\mathbb{H}(u\otimes \Psi_{\vec{f}})}
-2g^2\norm{ |k|^{-1/2}\chi(k)}^2_{L^2}\\
&=\underset{u\in\mathcal{U},\vec{f}_+\in L^2_\perp\cap\ZZ^+}{\inf}\;\underset{\vec{f}_-\in L^2_\perp\cap\ZZ^-}{\inf}\left(\E(u,\vec{A}_{\vec{f_+}+\vec{f_-}})+\ps{\vec{f}_-}{\abs{k}\vec{f}_-}_{L^2}\right)\\
& = \underset{u\in\mathcal{U},\vec{f}_+\in L^2_\perp\cap\ZZ^+}{\inf}\E(u,\vec{A}_{\vec{f_+}}),
\end{align*}
which establishes \eqref{eq:min_PF_Z+}.
\end{proof}
In the sequel we  focus on the minimization of the energy functional~$\E.$ By \eqref{eq:min_PF_Z+}, we can restrict the minimization to $\vec{f}\in\ZZ^+$.

In order for the coherent state $\Psi_{\vec{f}}$ to be well-defined, we assumed in the previous proof that $\vec{f}\in L^{2}_\perp(\mathbb{R}^{3};\mathbb{C}^{3})$. The further condition $\vec{f}\in\mathcal{Z}^+$ ensures that the term $\langle \vec{f},\abs{k}\vec{f}\rangle$ is finite. We will see below (see Lemma \ref{lem:coercivity-pauli-fierz}) that, in order for $\mathcal{E}(u,\vec{A}_{\vec{f}})$ to be well-defined, it suffices in fact to assume that $u\in\mathcal{U}$ and $\vec{f}\in\mathcal{Z}^+$. (By \eqref{eq:def_Af}, the latter condition is equivalent to~$\vec{A}_{\vec{f}}\in \dot{H}^{1}$, while $\vec{f}_+\in L^2$ is equivalent to $\vec{A}_{\vec{f}}\in\dot{H}^{1/2}$). 
We therefore study $\E$ on the energy space $\mathcal{U\times A}$ (where $\mathcal{A}$ is defined in \eqref{eq:def_space_A}), the norm on $\mathcal{U\times A}$ being defined by $$\|(u,\vec A)\|^2_{\mathcal{U\times A}} = \|u\|^2_{H^1} + \langle u,V_+ u\rangle_{L^2} + \|\vec A\|^2_{\dot H^1}.$$

In the remainder of this section, we establish Proposition \ref{prop:equiv_PF}, namely, that for any minimizer $(u_{\mathrm{gs}},\vec{A}_{\mathrm{gs}})$ in~$\mathcal{U}\times\mathcal{A}$ of the Maxwell--Schr\"odinger energy functional \eqref{eq:comput_E(u,A)}, there exists~$\vec{f}_{\mathrm{gs}}$ in~$L^{2}_\perp(\mathbb{R}^{3};\mathbb{C}^{3})\cap\ZZ$ such that~$\vec{A}_{\mathrm{gs}}=\vec{A}_{\vec{f}_\mathrm{gs}}$ as in \eqref{eq:def_Af}.

We begin with a lemma introducing the Euler-Lagrange equation satisfied by~$\vec{A}_{\mathrm{gs}}$ and the Pauli operator at a minimizer~$(u_\mathrm{gs}, \vec{A}_\mathrm{gs})$, which will often be useful in the sequel.
\begin{lemma}[Euler-Lagrange equation and Pauli operator associated to a minimizer]\label{lem:H_VAgs}
Suppose that the potential~$V$ satisfies Hypothesis~\ref{condVPauliFierz} and that~$\chi$ satisfies Hypothesis~\ref{condChi}.
If~$(u_\mathrm{gs}, \vec{A}_\mathrm{gs})$ is a minimizer of~$\mathcal{E}$ over~$\mathcal{U\times A}$, then
\begin{align}
\vec{A}_\mathrm{gs} &= 32\pi^3(-\Delta)^{-1}g{\hat{\chi}}*\Real
\big\langle (-i\vec{\nabla} +\vec{\nabla}\wedge\vec{\sigma} -g\hat{\chi}*\vec{A}_\mathrm{gs}) u_\mathrm{gs},u_{\mathrm{gs}}\big\rangle_{\mathbb{C}^2} , \label{eq:PF-Lagrange-A}
\end{align}
the operator
\begin{equation}\label{eq:def_H_VAgs}
H_{V,\vec{A}_\mathrm{gs}} := (-i\vec{\nabla}-g\hat{\chi} * \vec{A}_\mathrm{gs})^2 - g\hat{\chi}*\vec{\sigma} \cdot (\vec{\nabla}\wedge\vec{A}_{\mathrm{gs}})+ V+\frac{1}{32\pi^3}\|\vec{A}_\mathrm{gs}\|_{\dot{H}^1}^2 
\end{equation}
defines a self-adjoint operator, and $u_\mathrm{gs}$ is an eigenvector of $H_{V,\vec{A}_\mathrm{gs}}$ associated to the eigenvalue~$E_V$.
\end{lemma}
\begin{proof}
At a minimizer, the Frechet derivative of $\mathcal{E}(u,\vec{A})$ with respect to~$\vec{A}$,
\begin{align*}
\partial_{\vec{A}}{\E}(u_\mathrm{gs},\vec{A}_\mathrm{gs}) 
=& -\frac{\Delta}{16\pi^3}\vec{A}_\mathrm{gs}-g\hat{\chi}*2\Real\langle -i\vec{\nabla} u_\mathrm{gs},u_\mathrm{gs}\rangle_{\mathbb{C}^2} \nonumber\\
&+2g\hat{\chi}*[ (g\hat{\chi}*\vec{A}_\mathrm{gs})\abs{u_\mathrm{gs}}^2_{\mathbb{C}^2} ] 
-2g\hat{\chi}*\Real\langle \vec{\nabla}\wedge\vec{\sigma}u,u\rangle_{\mathbb{C}^2}\,, % \label{eq:PF-dAE}
\end{align*}
vanishes, which yields \eqref{eq:PF-Lagrange-A}.

Note that under our assumptions, $V_-$ is infinitesimally form bounded with respect to the operator~$(\vec{\sigma}\cdot(-i\vec{\nabla}-g\hat{\chi} * \vec{A}_\mathrm{gs}))^2$ (see \eqref{eq:V-_rel_form_bounded} below) from which, using the KLMN Theorem, it is not difficult to deduce that $H_{V,\vec{A}_\mathrm{gs}}$ identifies with a self-adjoint operator.
The minimizing property of~$(u_\mathrm{gs},\vec{A}_\mathrm{gs})$ means that
\begin{equation}\label{eq:HVA_min}
\langle u_\mathrm{gs}, (H_{V,\vec{A}_\mathrm{gs}}-E_V) u_\mathrm{gs} \rangle = 0\,.
\end{equation}
As 
$H_{V,\vec{A}_\mathrm{gs}}\geq E_V$, \eqref{eq:HVA_min} implies that $(H_{V,\vec{A}_\mathrm{gs}}-E_V)^{1/2} \, u_\mathrm{gs} = 0$ and thus
\begin{equation}\label{eq:PF-HVA-uc}
(H_{V,\vec{A}_\mathrm{gs}}-E_V) u_\mathrm{gs} = 0 \,,
\end{equation}
which ends the proof of the lemma.
\end{proof}

Two important ingredients in the proof of Proposition \ref{prop:equiv_PF} are the exponential decay of the electronic part $u_{\mathrm{gs}}$ and a virial argument. We begin with proving these two properties in Lemmata \ref{lm:exp_decay} and \ref{lm:virial}, respectively.
\begin{lemma}[Exponential decay of the ground state]\label{lm:exp_decay}
Under the assumptions of Proposition~\ref{prop:equiv_PF}, there exists $\gamma>0$ such that
\begin{equation}\label{eq:exp_decay}
\big\|e^{\gamma|x|}u_{\mathrm{gs}}\big\|_{L^2}<\infty.
\end{equation}
\end{lemma}
\begin{proof}
Recall from Lemma~\ref{lem:H_VAgs} that $u_{\mathrm{gs}}$ is a ground state of the Pauli operator~\eqref{eq:def_H_VAgs}. 
In particular, it is then known that $u_{\mathrm{gs}}$ decays exponentially in the sense that \eqref{eq:exp_decay} holds for some $\gamma>0$ (see e.g. \cite[Theorem 1]{Griesemer04}).
\end{proof}
\begin{lemma}[Virial argument]\label{lm:virial}
Under the assumptions of Proposition \ref{prop:equiv_PF}, 
\begin{equation*}
\big\langle u_{\mathrm{gs}},\big( -i\vec{\nabla}-g\hat\chi*\vec{A}_{\mathrm{gs}}\big) u_{\mathrm{gs}}\big\rangle=0.
\end{equation*}
\end{lemma}
\begin{proof}
We use the Pauli operator defined in~\eqref{eq:def_H_VAgs}. A direct computation shows that, in the sense of quadratic forms on $\mathcal{D}(H_{V,\vec{A}_{\mathrm{gs}}})\cap\mathcal{D}(x)$, we have
\begin{equation*}
\big[ H_{V,\vec{A}_{\mathrm{gs}}} , x \big ] = -2i \big( -i\vec{\nabla}-g\hat\chi*\vec{A}_{\mathrm{gs}}\big).
\end{equation*}
Since $H_{V,\vec{A}_{\mathrm{gs}}}u_{\mathrm{gs}}=E_Vu_{\mathrm{gs}}$ and $u_{\mathrm{gs}}$ belongs to~$\mathcal{D}(x)$ by Lemma \ref{lm:exp_decay}, we deduce that
\begin{multline*}
\big\langle u_{\mathrm{gs}}, \big( -i\vec{\nabla}-g\hat\chi*\vec{A}_{\mathrm{gs}}\big) u_{\mathrm{gs}}\big\rangle \\
= \frac{i}{2} \big\langle \big( H_{V,\vec{A}_{\mathrm{gs}}} - E_V \big ) u_{\mathrm{gs}},x u_{\mathrm{gs}}\big\rangle - \frac{i}{2} \big\langle x u_{\mathrm{gs}},\big( H_{V,\vec{A}_{\mathrm{gs}}} - E_V \big )  u_{\mathrm{gs}}\big\rangle = 0.
\end{multline*}
This proves the lemma.
\end{proof}
Now we are ready to prove Proposition \ref{prop:equiv_PF}.
\begin{proof}[Proof of Proposition \ref{prop:equiv_PF}]
Recall that $\vec{f}_{\mathrm{gs}}$ and $\vec{A}_{\vec{f}_{\mathrm{gs}}}$ are related as in \eqref{eq:def_Af}.
Moreover, $\vec{A}_{\mathrm{gs}}$ satisfies the relation \eqref{eq:PF-Lagrange-A}, which 
implies that
\begin{equation*}
\|\vec{f}_{\mathrm{gs}}\|_{L^2}\lesssim \big\|\vec{A}_{\vec{f}_{\mathrm{gs}}}\big\|_{\dot{H}^{1/2}}\lesssim |g| \Big\|\frac{\chi}{|k|^{\frac32}} \vec{F}_{\mathrm{gs}} \Big\|_{L^2},
\end{equation*}
where, to shorten notations, we have set $\vec{F}_{\mathrm{gs}}:=\vec{F}_{\mathrm{gs},1}+\vec{F}_{\mathrm{gs},2}$, with
\begin{equation*}
\vec{F}_{\mathrm{gs},1}:=\bar{\mathcal{F}}\big( \langle -i\vec{\nabla}u_{\mathrm{gs}}-g(\hat\chi*\vec{A}_{\vec{f}_{\mathrm{gs}}}) u_{\mathrm{gs}}, u_{\mathrm{gs}}\rangle_{\mathbb{C}^2}\big), \quad \vec{F}_{\mathrm{gs},2}:=\bar{\mathcal{F}}\big(\langle \vec{\nabla}\wedge\vec{\sigma}u_{\mathrm{gs}},u_{\mathrm{gs}}\rangle_{\mathbb{C}^2} \big) .
\end{equation*}
We can estimate
\begin{align*}
\|\vec{f}_{\mathrm{gs}}\|_{L^2}&\lesssim \Big\|\frac{\chi}{|k|}\Big\|_{L^2+L^{3,\infty}}\Big\|\frac{1}{|k|^{\frac12}} \vec{F}_{\mathrm{gs}} \Big\|_{L^\infty \cap L^{6,2}} \,.
\end{align*}
Using the cutoff functions $\eta$, $\tilde{\eta}$, we separate the contributions from $k$ in a neighborhood of the origin and $k$ in a neighborhood of $\infty$, obtaining, since ${\tilde{\eta}}^2|k|^{-1/2}\le1$,
\begin{align*}
\|\vec{f}_{\mathrm{gs}}\|_{L^2}&\lesssim \Big\|\frac{\chi}{|k|}\Big\|_{L^2+L^{3,\infty}} \Big(\Big\|\frac{1}{|k|^{\frac12}} \eta^2\vec{F}_{\mathrm{gs}}\Big\|_{L^\infty \cap L^{6,2}}+\big\|\vec{F}_{\mathrm{gs}}\big\|_{L^\infty \cap L^{6,2}}\Big)\,.
\end{align*}
Clearly, $\|  \vec{F}_{\mathrm{gs}} \|_{L^\infty}<\infty$ since $F_{\mathrm{gs},1}$, $F_{\mathrm{gs},2}$ are the Fourier transforms of products of $L^2$-functions. Moreover, $\|\vec{F}_{\mathrm{gs}}\|_{L^{6,2}}<\infty$ by Lemma~\ref{lem:estimateChikFu1u2}. Thanks to the cutoff function $\eta$, we also have
\begin{equation*}
\Big\|\frac{1}{|k|^{\frac12}} \eta^2\vec{F}_{\mathrm{gs}}\Big\|_{L^{6,2}}\lesssim\Big\|\frac{1}{|k|^{\frac12}} \eta^2\vec{F}_{\mathrm{gs}}\Big\|_{L^{\infty}}.
\end{equation*}
Hence it remains to show that the right-hand-side of the previous equation is finite.

To this end, we estimate the contributions from $F_{\mathrm{gs},1}$ and $F_{\mathrm{gs},2}$ separately. We begin with~$F_{\mathrm{gs},1}$. We observe that, by Lemma \ref{lm:virial},
\begin{equation*}
\vec{F}_{\mathrm{gs},1}(0)=\big\langle u_{\mathrm{gs}},\big( -i\vec{\nabla}-g\hat\chi*\vec{A}_{\mathrm{gs}}\big) u_{\mathrm{gs}}\big\rangle_{L^2}=0.
\end{equation*}
Moreover, using Lemma \ref{lm:exp_decay}, Lemma \ref{lem:estimateChiAU} and the fact that $u_{\mathrm{gs}}$ belongs to~$\dot{H}^1$, we have, for all multi-index $\alpha\in\mathbb{N}^3$,
\begin{align*}
\big\|\partial_k^\alpha\vec{F}_{\mathrm{gs},1}\big\|_{L^\infty} &\lesssim \big\| \langle -i\vec{\nabla}u_{\mathrm{gs}}-(g\hat\chi*\vec{A}_{\vec{f}_{\mathrm{gs}}}) u_{\mathrm{gs}},  x^\alpha u_{\mathrm{gs}} \rangle_{\mathbb{C}^2}\big\|_{L^1}\\
&\lesssim \big\|  -i\vec{\nabla}u_{\mathrm{gs}}-(g\hat\chi*\vec{A}_{\vec{f}_{\mathrm{gs}}}) u_{\mathrm{gs}} \big\|_{L^2} \big\|x^\alpha u_{\mathrm{gs}} \big\|_{L^2} < \infty.
\end{align*}
Hence $\vec{F}_{\mathrm{gs},1}$ belongs to the Sobolev space $W^{\infty,\infty}(\mathbb{R}^3;\mathbb{R}^3)$. Applying the mean-value theorem then yields
\begin{equation*}
\Big\|\frac{1}{|k|^{\frac12}} \eta^2\vec{F}_{\mathrm{gs},1}\Big\|_{L^{\infty}}\le \sup_{|\alpha|=1}\big\||k|^{\frac12} \eta^2 \partial_x^\alpha \vec{F}_{\mathrm{gs},1}\big\|_{L^{\infty}}\lesssim \sup_{|\alpha|=1}\big\| \partial_x^\alpha \vec{F}_{\mathrm{gs},1}\big\|_{L^{\infty}}<\infty.
\end{equation*}

Now we consider $F_{\mathrm{gs},2}$. Since
\begin{align*}
\vec{F}_{\mathrm{gs},2}(k) = k\wedge\bar{\mathcal{F}}\big(\langle \vec{\sigma}u_{\mathrm{gs}},u_{\mathrm{gs}}\rangle_{\mathbb{C}^2} \big)(k),
\end{align*}
we can estimate
\begin{equation*}
\Big\|\frac{1}{|k|^{\frac12}} \eta^2\vec{F}_{\mathrm{gs},2}\Big\|_{L^{\infty}}\le \Big\||k|^{\frac12} \eta^2\bar{\mathcal{F}}\big(\langle \vec{\sigma}u_{\mathrm{gs}},u_{\mathrm{gs}}\rangle_{\mathbb{C}^2} \big)\Big\|_{L^{\infty}} \lesssim \Big\|\bar{\mathcal{F}}\big(\langle \vec{\sigma}u_{\mathrm{gs}},u_{\mathrm{gs}}\rangle_{\mathbb{C}^2} \big)\Big\|_{L^{\infty}}<\infty.
\end{equation*}
This concludes the proof of the proposition.
\end{proof}

\subsection{Coercivity, energy gap and existence of a minimizer}\label{sec:PF-existence}

In this section we prove Theorem \ref{thm:Pauli-Fierz}, namely the existence of a global minimizer for the Maxwell--Schr\"odinger energy functional. We use coercivity and lower semicontinuity arguments. 

Before we prove Theorem \ref{thm:Pauli-Fierz}, we establish a coercivity result which will allow us to show that any minimizing sequence is bounded in $\mathcal{U}\times\mathcal{A}$ (recall that $\mathcal{U}$ has been defined in \eqref{eq:defU_intro} and $\mathcal{A}$ in \eqref{eq:def_space_A}).
\begin{lemma}[Coercivity]
\label{lem:coercivity-pauli-fierz}
Suppose that $V$ satisfies Hypothesis~\ref{condVPauliFierz} and $\chi=\chi_1+\chi_2$ satisfies Hypothesis \ref{condChi}, with~$\chi_1/|k|$ in~$L^2$ and $\chi_2/|k|$ in~$L^{3,\infty}$. If 
\begin{equation*} %\label{eq:Cond_chi2-lemma}
32\pi^3 a C^2 g^2 \Big\|\frac{\chi_2}{|k|}\Big\|_{L^{3,\infty}}^2<1 \,,
\end{equation*}
with the constant $a\ge0$ from Hypothesis~\ref{condVPauliFierz} and the universal constant $C>0$ from Lemma~\ref{lem:estimateChiAU},
then for all $(u,\vec{A})$ in $\mathcal{U\times A}$ such that $\|(u,\vec{A})\|_{\mathcal{U\times A}}\geq 16(2+a)^{2}$\, we have
\[
\mathcal{E}_V(u,\vec{A})\geq C_{1}\|(u,\vec{A})\|_{\mathcal{U\times A}} - C_{2}\, ,
\]
with
\begin{itemize}
    \item $C_{1}=:\varepsilon/\max\{4,32g^{2}C^2\big\| \chi/|k|\big\|_{L^2+L^{3,\infty}}^2\}$, 
    \item $C_2=:b + a^{2}\big(1+\frac{C^2g^2}{\varepsilon} \big\| \chi_1/|k| \big\|_{L^2}^2\big)$, 
    \item $2\varepsilon :=(32\pi^3)^{-1}- C^2ag^{2} \big\| \chi_2/|k| \big\|_{L^{3,\infty}}^{2}$.
\end{itemize}
\end{lemma}

\begin{proof}
Thanks to Lemma~\ref{lem:estimateChiAU}, with the constant~$a$ from Hypothesis~\ref{condVPauliFierz}, we can write
\begin{multline*}
\|u\|_{\dot{H}^{1/2}}^{2} 
\leq \|\vec{\nabla} u\|_{L^{2}} 
= \|\vec{\sigma} \cdot \vec{\nabla} u\|_{L^{2}} 
\leq \|\vec{\sigma} \cdot(-i\vec{\nabla}-g\,\hat \chi*\vec{A})u\|_{L^{2}}+g\|\vec{\sigma} \cdot(\hat \chi*\vec{A})\,u\|_{L^{2}}\\
\leq \|\vec{\sigma} \cdot(-i\vec{\nabla}-g\,\hat \chi*\vec{A})u\|_{L^{2}}+Cg \|\vec{A}\|_{\dot{H}^{1}} \left(\left\|\frac{\chi_1}{|k|}\right\|_{L^2}+\left\|\frac{\chi_2}{|k|}\right\|_{L^{3,\infty}}\|u\|_{\dot{H}^{1/2}} \right)\\
\leq\|\vec{\sigma} \cdot (-i\vec{\nabla}-g\,\hat \chi*\vec{A})u\|_{L^{2}}+ \frac{aC^2g^2}{2\varepsilon} \left\|\frac{\chi_1}{|k|}\right\|_{L^2}^2 
+\left(\frac{\varepsilon}{2a}+\frac{C^2g^{2}}{2}\left\|\frac{\chi_2}{|k|}\right\|_{L^{3,\infty}}^{2} \right) \|\vec{A}\|_{\dot{H}^{1}}^{2}
+\frac{1}{2}\|u\|_{\dot{H}^{1/2}}^{2}\,.
\end{multline*}
Hence, 
\begin{align*}
\|u\|_{\dot{H}^{1/2}}^{2}\leq2\|\vec{\sigma} \cdot(-i\vec{\nabla}-g\,\hat\chi*\vec{A})u\|_{L^{2}}+\frac{a C^2 g^2}{\varepsilon} \left\|\frac{\chi_1}{|k|}\right\|_{L^2}^2 
+\left(\frac{\varepsilon}{a}+C^2 g^{2} \left\|\frac{\chi_2}{|k|}\right\|_{L^{3,\infty}}^{2} \right) \|\vec{A}\|_{\dot{H}^{1}}^{2} \,.
\end{align*}
It follows from Hypothesis \ref{condVPauliFierz} that 
\begin{align}
\langle u,V_{-} u\rangle 
\leq 2a\|\vec{\sigma} \cdot(-i\vec{\nabla}-g\,\hat\chi*\vec{A})u\|_{L^{2}}
+\Big(\frac{1}{32\pi^3}-\varepsilon\Big) \|\vec{A}\|_{\dot{H}^{1}}^{2} 
+\frac{C^2 a^2 g^2}{\varepsilon} \left\|\frac{\chi_1}{|k|}\right\|_{L^2}^2  + b \label{eq:V-_rel_form_bounded} 
\end{align}
and hence,
\begin{align*}
\mathcal{E}_V(u,\vec A) & \geq \langle u,V_{+} u\rangle + \big(\|\vec{\sigma} \cdot(-i\vec{\nabla}-g\,\hat\chi*\vec{A})u\|_{L^{2}}-a\big)^{2}
 +\varepsilon \|\vec{A}\|_{\dot{H}^{1}}^{2}\\
 &\qquad-a^{2}\Big(1+\frac{C^2 g^2}{\varepsilon} \left\|\frac{\chi_1}{|k|}\right\|_{L^2}^2\Big) - b\\
& \geq \varepsilon\Big(\langle u,V_{+} u\rangle + \big(\|\vec{\sigma} \cdot(-i\vec{\nabla}-g\,\hat\chi*\vec{A})u\|_{L^{2}}-a\big)^{2}+\|\vec{A}\|_{\dot{H}^{1}}^{2}\Big) - C_2.
\end{align*}
Let us suppose that $R=\|(u,\vec{A})\|_{\mathcal{U\times A}}\geq4$. We consider
three cases:
\begin{enumerate}[leftmargin=0.7cm]
\item If $\|\vec{A}\|_{\dot{H}^{1}}\geq R/4$, then $\mathcal{E}(u,\vec{A})\geq \varepsilon R^{2}/16 - C_2\geq \varepsilon R/4 - C_2$.
\item If $\langle u,V_{+} u\rangle \geq R^2/16$, then $\mathcal{E}(u,\vec{A})\geq \varepsilon R^{2}/16 - C_2 \geq \varepsilon R/4 - C_2$.
\item Otherwise $\|u\|_{H^1}\geq R/2$ and
$$\|\vec{\nabla} u\|_{L^{2}}^2 =\|u\|_{H^{1}}^2-1 \geq R^2/4-1\geq (R/2-1)^2\,.$$
We distinguish two subcases:
\begin{enumerate}[leftmargin=0.7cm]
\item If $\big(\|\vec{\sigma} \cdot(-i\vec{\nabla}-g\,\hat\chi*\vec{A})u\|_{L^{2}}-a\big)^{2} \geq R/4$, then $\mathcal{E}(u,\vec{A})\geq \varepsilon R/4-C_2$, 
\item If $\big(\|\vec{\sigma} \cdot(-i\vec{\nabla}-g\,\hat\chi*\vec{A})u\|_{L^{2}}-a\big)^{2} < R/4$, then
\begin{align*}
\frac{R}{2}-1-a-\frac{R^{1/2}}{2}&\leq\|\vec{\nabla} u\|_{L^{2}}-\|\vec{\sigma}\cdot(-i\vec{\nabla} -g\,\hat\chi*\vec{A})u\| \\
&\leq g\|(\hat\chi*\vec{A})u\|_{L^{2}}\\
&\leq 4gC\Big\|\frac{\chi}{|k|}\Big\|_{L^2+L^{3,\infty}}\|\vec{A}\|_{\dot{H}^{1}} \frac{R^{1/2}}{2}.
\end{align*}
Therefore, for $R\geq 16(2+a )^{2}$,
\[
\frac{R^{1/2}}{4\sqrt{2}gC\big\|\frac{\chi}{|k|}\big\|_{L^2+L^{3,\infty}}}\leq\frac{R^{1/2}-(2+a)}{4gC\big\|\frac{\chi}{|k|}\big\|_{L^2+L^{3,\infty}}}\leq\|\vec{A}\|_{\dot{H}^{1}}
\]
 and hence
\[
\mathcal{E}_V(u,\vec{A})\geq \varepsilon \|\vec{A}\|_{\dot{H}^{1}}^{2}-C_2 \geq\frac{\varepsilon}{32g^{2}C^2\big\|\frac{\chi}{|k|}\big\|_{L^2+L^{3,\infty}}^2}R-C_2\,.
\]
\end{enumerate}
\end{enumerate}
This yields the result. 
\end{proof}
We are now ready to prove Theorem~\ref{thm:Pauli-Fierz}.
\begin{proof}[Proof of Theorem~\ref{thm:Pauli-Fierz}]
Let $(u_{j},\vec A_{j})_{j\in\mathbb{N}}$
be a minimizing sequence for $\mathcal{E}$ in $\mathcal{U\times A}$. In particular, $(\mathcal{E}(u_{j},\vec A_{j}))_j$ is bounded and hence, by Lemma \ref{lem:coercivity-pauli-fierz}, $(u_{j},\vec A_{j})_j$ is bounded in $\mathcal{U\times A}$. Hence the sequence $(u_{j},\vec A_{j})_{j}$ converges weakly
to some limit $(u_{\infty},\vec A_{\infty})$ in~$\mathcal{U\times A}$ w.r.t.~the topology of~$\mathcal{Q}_V\times \dot{H}^1$.

We first show that 
\begin{equation}
\|u_{j}-u_{\infty}\|_{L^{2}}\xrightarrow[j\to\infty]{}0\,.\label{eq:uj-to-u-infinity}
\end{equation}
Let $\varepsilon>0$. By Hypothesis~\ref{condVPauliFierz} there exists~$R>0$
such that $|V_{2}(x)|\leq\varepsilon(E_{V_{1}}-E_{V})$ for $|x|\geq R$.
Recall that the cutoff functions $\eta_R$, $\tilde{\eta}_R$ have been defined in \eqref{eq:defetaR}.
We have
\begin{multline}
\|u_{j}-u_{\infty}\|_{L^{2}}^{2}=\|\eta_{R}(u_{j}-u_{\infty})\|_{L^{2}}^{2}+\|\tilde{\eta}_{R}(u_{j}-u_{\infty})\|_{L^{2}}^{2}\\
\leq\|\eta_{R}(u_{j}-u_{\infty})\|_{L^{2}}^{2}+2\|\tilde{\eta}_{R}u_{j}\|_{L^{2}}^{2}+2\|\tilde{\eta}_{R}u_{\infty}\|_{L^{2}}^{2}\,.\label{eq:splitting-of-uj-uinfty}
\end{multline}
\emph{}For $u$ in~$\mathcal{U}$, the magnetic IMS localization formula~(\ref{eq:IMS-magnetic}) yields
\begin{align}
\mathcal{E}_V(u,\vec A) & =\langle\eta_{R}\,u,((-i\vec{\nabla}-\hat\chi * \vec{A}\, )^{2}+V - \vec{\sigma} \cdot (\vec{\nabla}\wedge \vec{A}))\, \eta_{R}\,u\rangle \nonumber \\
 & \qquad+\langle\tilde{\eta}_{R}\,u,((-i\vec{\nabla}-\hat\chi * \vec{A}\, )^{2}+V_{1} - \vec{\sigma} \cdot (\vec{\nabla}\wedge \vec{A}))\, \tilde{\eta}_{R}\,u\rangle \nonumber \\
 & \qquad+\langle u,(\tilde{\eta}_{R}^{2}\,V_{2}+|\vec{\nabla}\eta_{R}|^{2}+|\vec{\nabla}\tilde{\eta}_{R}|^{2})u\rangle+\frac{1}{32\pi^3}\|\vec A\|^2_{\dot{H}^1} \nonumber \\
 & =\mathcal{E}_V(\frac{\eta_{R}\,u}{\|\eta_{R}\,u\|_{L^{2}}}, \vec A )\,\|\chi_{R}\,u\|_{L^{2}}^{2}+\mathcal{E}_{V_{1}}(\frac{\tilde{\eta}_{R}\,u}{\|\tilde{\eta}_{R}\,u\|_{L^{2}}}, \vec A )\,\|\tilde{\eta}_{R}\,u\|_{L^{2}}^{2}\nonumber \\
 & \qquad+\langle u,(\tilde{\eta}_{R}^{2}\,V_{2}+|\vec{\nabla}\eta_{R}|^{2}+|\vec{\nabla}\tilde{\eta}_{R}|^{2})u\rangle\nonumber \\
 & \geq E_{V}\,\|\eta_{R}\,u\|_{L^{2}}^{2}+E_{V_{1}}\,\|\tilde{\eta}_{R}\,u\|_{L^{2}}^{2}-\varepsilon(E_{V_{1}}-E_{V}) . \label{eq:lower-bound-E-tilde-control-chi-R-bar-u}
\end{align}
As $(u_{j}, \vec A_{j})$ is a minimizing sequence, (\ref{eq:lower-bound-E-tilde-control-chi-R-bar-u})
yields, for $j$ large enough,
\begin{equation}
\|\tilde{\eta}_{R}\,u_{j}\|_{L^{2}}^{2}\leq\frac{\mathcal{E}_V(u_{j},\vec A_{j})-E_{V}}{E_{V_{1}}-E_{V}}+\varepsilon\leq2\varepsilon\,.\label{eq:chi-R-bar-u-j}
\end{equation}
By the lower semi-continuity of the $L^{2}$ norm,
\begin{equation}
\|\tilde{\eta}_{R}\,u_{\infty}\|_{L^{2}}^{2}\leq\liminf_{j\to\infty}\|\tilde{\eta}_{R}\,u_{j}\|_{L^{2}}^{2}\leq2\varepsilon\,.\label{eq:chi-r-u-infty}
\end{equation}
The sequence $(u_{j})_{j}$ converges towards $u_{\infty}$ weakly
in $\mathcal{Q}_V$ and thus the sequence $(\eta_{R}\,u_{j})$ converges
weakly in $H^{1}$ towards $\eta_{R}\,u_{\infty}$. Using the
compactness of the set $\overline{B(0,2R)}$ and the Rellich-Kondrachov
Theorem, we deduce that $(\eta_{R}\,u_{j})$ converges strongly in $L^{2}$ to
$\eta_{R}\,u_{\infty}$. This together with (\ref{eq:splitting-of-uj-uinfty}),
(\ref{eq:chi-R-bar-u-j}) and (\ref{eq:chi-r-u-infty}) prove (\ref{eq:uj-to-u-infinity}).

To show that ${\displaystyle \liminf_{j\to\infty}\mathcal{E}(u_{j},\vec A_{j})\geq\mathcal{E}(u_{\infty},\vec A_{\infty})}$,
we split $\mathcal{E}(u,\vec A\, )$ into five parts:
\begin{multline*}
\mathcal{E}_V(u,\vec A)=
\overset{\mathcal{E}_{1}(u)}{\overbrace{\langle u,H_V u\rangle}}+\frac{1}{32\pi^3}\overset{\mathcal{E}_{2}(\vec A)}{\overbrace{\|\vec A\|^2_{\dot{H}^1}}}
-2g \Real\overset{\mathcal{E}_{3}(u,\vec A)}{\overbrace{\big\langle -i\vec{\nabla} u, (\hat\chi*\vec{A}) u \big\rangle}} \\
+\underset{\mathcal{E}_{4}(u,\vec A)}{\underbrace{\big\langle u,(\hat\chi*\vec{A})^{2}u\big\rangle}}
-g\underset{\mathcal{E}_{5}(u,\vec A)}{\underbrace{ \big\langle u,\vec{\sigma}\cdot (\hat\chi * \vec{\nabla} \wedge \vec{A}) u \big\rangle}} \,.
\end{multline*}
By Lemma~\ref{lem:LowerSemiContDeltaV}, we have
\begin{equation}
\liminf_{j\to\infty}\mathcal{E}_{1}(u_{j})\geq\mathcal{E}_{1}(u_{\infty})\label{eq:lim-E-tilde-1}\,.
\end{equation}
By the lower semi-continuity of $\|\cdot\|_{\dot{H}^1}$,
\begin{equation}
\liminf_{j\to\infty}\mathcal{E}_{2}(\vec A_{j})\geq\mathcal{E}_{2}(\vec A_{\infty})\,.\label{eq:lim-E-tilde-2}
\end{equation}
Now using $\vec \nabla \cdot \vec{A}_j=0$, the Cauchy-Schwarz inequality, the boundedness of $(\|(u_{j},\vec A_{j})\|_{\mathcal{U\times A}})_{j}$ and Lemma~\ref{lem:estimateChiAU}, we obtain
\begin{equation}
|\mathcal{E}_{3}(u_{j},\vec A_{j})
- \big\langle -i\vec{\nabla}u_{j}, (\hat\chi*\vec{A}_j ) \, u_{\infty}\big\rangle 
+ \big\langle  (\hat\chi*\vec{A}_j)  \, u_{j}, \, -i\vec{\nabla} u_{\infty}\big\rangle 
- \mathcal{E}_{3}(u_{\infty},\vec A_{j})|
\lesssim \|u_{j}-u_{\infty}\|^{\frac12}_{L^{2}}. \label{eq:E2tilde-CV-1}
\end{equation}
We claim that the weak convergence of $\vec A_j$ towards $\vec A_\infty$ then yields
\begin{equation}
\mathcal{E}_{3}(u_{\infty},\vec A_{j})\xrightarrow[j\to\infty]{}\mathcal{E}_{3}(u_{\infty},\vec A_{\infty}) \, . \label{eq:E2tilde-CV-3}
\end{equation}
The limit \eqref{eq:E2tilde-CV-3} can be proven as follows. 
 Let $\varphi  = \overline{-i\vec \nabla u_\infty }$. Then
\begin{equation}\label{eq:E3uInftyAJ}
\mathcal{E}_{3}( u_\infty ,\vec A_{j})
= \int \left(\hat \chi * (u_\infty \varphi )\right) \, \vec A_j 
\end{equation}
and it thus suffices to verify that
$\hat \chi * (u_\infty \varphi )$ belongs to $\dot{H}^{-1}$. This is a consequence of Lemma~\ref{lem:estimateChikFu1u2}:
\begin{equation}
\big\|\hat \chi  * (u_\infty \varphi ) \big\|_{\dot{H}^{-1}} 
\lesssim \Big\| \frac{\chi}{|k|} \Big\|_{L^2+L^{3,\infty}} \|u_\infty\|_{H^1}\| \varphi\|_{L^2}
\lesssim \Big\| \frac{\chi}{|k|} \Big\|_{L^2+L^{3,\infty}} \|u_\infty\|_{H^1}^2<\infty\,.
\end{equation}
The bound (\ref{eq:E2tilde-CV-1}), and the limit (\ref{eq:E2tilde-CV-3})
yield 
\begin{equation}
\mathcal{E}_{3}(u_{j},\vec A_{j})\xrightarrow[j\to\infty]{}\mathcal{E}_{3}(u_{\infty}, \vec A_{\infty})\,. \label{eq:lim-E-tilde-3}
\end{equation}
Similarly as in \eqref{eq:E2tilde-CV-1}, we have
\begin{equation}
|\mathcal{E}_{4}(u_{j},\vec A_{j})-\big\langle u_{\infty},\vec{A}_j^{2}u_{j}\big\rangle_{L^{2}}+\big\langle u_{\infty},\vec{A}_j^{2}u_{j}\big\rangle_{L^{2}}-\mathcal{E}_{4}(u_{\infty},\vec A_{j})| \lesssim \|u_j-u_\infty\|^{\frac12}_{L^2} \,. \label{eq:lim-E-tilde-4}
\end{equation}
Let us now prove that
\begin{equation}
\liminf_{j\to\infty}\mathcal{E}_{4}(u_{\infty},\vec A_{j})\geq \mathcal{E}_{4}(u_{\infty},\vec A_{\infty})\,. \label{eq:E4UInftyAJ}
\end{equation}
Using the same arguments as those used to prove the lower semicontinuity of norms, we can write
\begin{equation}
\mathcal{E}_{4}(u_{\infty},\vec A_{j}-\vec A_\infty) = \mathcal{E}_{4}(u_{\infty},\vec A_{j}) +\mathcal{E}_{4}(u_{\infty},\vec A_\infty) -2\Real\left\langle (\hat \chi * \vec A_\infty ) u_\infty,  (\hat \chi * \vec A_j ) u_\infty  \right\rangle \geq 0 \label{eq:E4Positivity}
\end{equation}
and arguing as in \eqref{eq:E3uInftyAJ}, with $\varphi = \overline{(\hat \chi * \vec A_\infty ) u_\infty} $ (which belongs to $L^2$ by Lemma~\ref{lem:estimateChiAU}), we deduce that
\begin{equation}
\left\langle (\hat \chi * \vec A_\infty ) u_\infty,  (\hat \chi * \vec A_j ) u_\infty \right\rangle
=\int \left( \hat \chi * \left( u_\infty \varphi \right)\right)  \vec A_j  \xrightarrow[j\to\infty]{} \mathcal{E}_{4}(u_{\infty}, \vec A_\infty) \,. \label{eq:E4WeakConvergence}
\end{equation}
Now \eqref{eq:E4Positivity}-\eqref{eq:E4WeakConvergence} imply \eqref{eq:E4UInftyAJ}.
A convenient expression of the last term,
\begin{equation*}
\mathcal{E}_{5}(u,\vec A) 
= \big\langle u,\vec{\sigma}\cdot (\hat\chi * \vec{\nabla} \wedge \vec{A}) u \big\rangle
= -2\Real  \big\langle  \vec{\nabla} \wedge \vec{\sigma}u\,, (\hat\chi * \vec{A}) u \big\rangle,
\end{equation*}
shows that it can be handled as $\mathcal{E}_3$ and
\begin{equation*}
\mathcal{E}_{5}(u_{j},\vec A_{j})\xrightarrow[j\to\infty]{}\mathcal{E}_{5}(u_{\infty}, \vec A_{\infty} )\,.%\label{eq:lim-E-tilde-5}
\end{equation*}
Finally, \eqref{eq:lim-E-tilde-1}, \eqref{eq:lim-E-tilde-2},
\eqref{eq:lim-E-tilde-3}, \eqref{eq:lim-E-tilde-4}, \eqref{eq:E4UInftyAJ} and \eqref{eq:lim-E-tilde-3} imply that
\begin{equation*}
 \liminf_{j\to\infty}\mathcal{E}(u_{j},\vec A_{j})\geq\mathcal{E}(u_{\infty},\vec A_{\infty})
 \end{equation*}
and hence the infimum is indeed a minimum, since $(u_{\infty},\vec A_{\infty})$
is a minimizer.
\end{proof}
To conclude this subsection, we focus on the condition $E_{V_1}>E_V$ which was a crucial assumption in our proof of the existence of a minimizer in Theorem~\ref{thm:Pauli-Fierz}. As mentioned in Remark \ref{rk:cond_gap}, the next proposition shows that this condition is satisfied provided that~$|g| \,\|\chi_2/|k|\|_{L^{3,\infty}}$ is not too large and that either $V$ is confining (recall from Lemma~\ref{lm:confining_intro} that in this case $\mu_{V_1}$ can be chosen arbitrarily large) or $\mu_{V_1}>\mu_V$ and $|g| \,\|\chi/|k|\|_{L^2+L^{3,\infty}}$ is small enough.

\begin{proposition}[Existence of a gap]
\label{lem:delta-E-geq-delta-e}
Suppose that $V$ satisfies Hypothesis~\ref{condVPauliFierz} and that the cut-off function~$\chi=\chi_1+\chi_2$ satisfies Hypothesis \ref{condChi} with $\chi_1/|k|$ in~$L^2$ and $\chi_2/|k|$ in~$L^{3,\infty}$.
If $\mu_V\geq 0$, then there exists a positive constant~$C_V$ such that, for all~$0<\beta<1-C_V g \left\| \chi_2/|k| \right\|_{L^{3,\infty}}$,
\[
E_{V_{1}}-E_V\geq \min\Big(1,\Big(1-\beta-C_Vg\left\|\frac{\chi_2}{|k|}\right\|_{L^{3,\infty}}\Big)\mu_{V_1}- \mu_V-\frac{C_V^2 g^2}{4\beta}\left\|\frac{\chi_1}{|k|}\right\|_{L^2}^2 -C_Vg\left\|\frac{\chi_2}{|k|}\right\|_{L^{3,\infty}} \Big).
\]
\end{proposition}
\begin{proof}
Restricting the infimum of $\mathcal{E}_V$ to $\mathcal U \times \{\vec{0}\}$ yields an upper bound for $E_V$:
\begin{equation}\label{eq:upper-bound-EV-1}
E_V
= \inf_{(u,\vec{A}) \in\mathcal{U}\times \mathcal{A}} \mathcal{E}_V(u,\vec{A})
\leq \inf_{u \in\mathcal{U}} \mathcal{E}_V(u,\vec{0})
= \mu_V \,.
\end{equation}
To control $E_{V_1}$ from below, recall that
\[
\mathcal{E}_{V_1}(u,\vec{A}) =
\| \vec{\sigma}\cdot (-i\vec{\nabla} - g\hat\chi * \vec{A}) \, u \|^2_{L^{2}} + \langle u, V_1 u\rangle
+\frac{1}{32\pi^3}\| \vec A \|^2_{\dot{H}^1}.
\]
If $\| \vec A \|^2_{\dot{H}^1} / (32\pi^3)\ge \mu_V+1$, then
$$
\mathcal{E}_{V_1}(u,\vec{A}) \ge \mu_V+1.
$$
Suppose now that $\| \vec A \|^2_{\dot{H}^1} / (32\pi^3) \le \mu_V+1$. Then, with $C>0$ the universal constant from Lemma~\ref{lem:estimateChiAU}, and $C_V = 2C\sqrt{\mu_V + 1}$,
\begin{align*}
\mathcal{E}_{V_{1}}(u,\vec A)
&\ge\langle u,(-\Delta_{x}+V_{1})u\rangle
 -2g\big|\big\langle  \vec{\sigma} \cdot \vec{\nabla}u\,, \vec{\sigma}\cdot(\hat\chi * \vec{A}) u \big\rangle\big|\\ \allowdisplaybreaks
&\ge\langle u,(-\Delta_{x}+V_{1})u\rangle
- 2gC \|u\|_{\dot{H}^{1}} \|\vec{A}\|_{\dot{H}^{1}} \left(\left\|\frac{\chi_1}{|k|}\right\|_{L^2}+\left\|\frac{\chi_2}{|k|}\right\|_{L^{3,\infty}}\|u\|_{\dot{H}^{1/2}} \right)\\ \allowdisplaybreaks
&\ge\langle u,(-\Delta_{x}+V_{1})u\rangle
- gC_V \left(\left\|\frac{\chi_1}{|k|}\right\|_{L^2}\|u\|_{\dot{H}^{1}}+\left\|\frac{\chi_2}{|k|}\right\|_{L^{3,\infty}}\|u\|_{H^1}^{2} \right) \,.
\end{align*}
Now, thanks to the conditions on $\beta$,
\begin{align*}
\mathcal{E}_{V_{1}}(u,\vec A)
&\ge\Big(1-\beta - gC_V \left\|\frac{\chi_2}{|k|}\right\|_{L^{3,\infty}}\Big)\langle u,(-\Delta_{x}+V_{1})u\rangle 
-\frac{C_V^2 g^2}{4\beta}\left\|\frac{\chi_1}{|k|}\right\|_{L^2}^2 -  gC_V\left\|\frac{\chi_2}{|k|}\right\|_{L^{3,\infty}} \\ \allowdisplaybreaks
&\ge\Big(1-\beta - gC_V \left\|\frac{\chi_2}{|k|}\right\|_{L^{3,\infty}}\Big)\mu_{V_1} 
-\frac{C_V^2 g^2}{4\beta}\left\|\frac{\chi_1}{|k|}\right\|_{L^2}^2 -  gC_V\left\|\frac{\chi_2}{|k|}\right\|_{L^{3,\infty}} \,.
\end{align*}
Therefore
\[
E_{V_{1}}\geq \min\Big(\mu_V+1,\Big(1-\beta-C_Vg\left\|\frac{\chi_2}{|k|}\right\|_{L^{3,\infty}}\Big)\mu_{V_1}- \frac{C_V^2 g^2}{4\beta}\left\|\frac{\chi_1}{|k|}\right\|_{L^2}^2 -C_Vg\left\|\frac{\chi_2}{|k|}\right\|_{L^{3,\infty}} \Big),
\]
which together with (\ref{eq:upper-bound-EV-1}) 
yields the result.
\end{proof}

\subsection{Properties of the set of minimizers}\label{sec:PF-uniqueness}

In this section we prove some estimates on minimizers $(u_{\mathrm{gs}},\vec{A}_{\mathrm{gs}})$ of the Maxwell--Schr\"odinger energy functional, which in turn implies that~$\vec{A}_{\mathrm{gs}}$ is fully determined by~$u_{\mathrm{gs}}$ for small~$g$.

We use the following notations. The resolvent of the operator $H_V\otimes \mathbf{I}_{\mathbb{C}^2}$ is denoted by~$R_\lambda:=(H_V-\lambda)^{-1}  \otimes \mathbf{I}_{\mathbb{C}^2}$ (a priori defined as an unbounded operator on the set~$\mathrm{Ran}(\mathds{1}_{\{\lambda\}}(H_V\otimes \mathbf{I}_{\mathbb{C}^2}))^{\perp}$). Recall that $\Pi_V$, $\Pi_V^\perp$ are the projections in $L^2(\mathbb{R}^3;\mathbb{C}^2)$ defined by $\Pi_V:=\ket{u_V}\bra{u_V} \otimes \mathbf{I}_{\mathbb{C}^2}$, $\Proj = \mathbf{I}_{L^2}-\Pi_V$. Moreover, for all~$u$ in~$\mathcal{Q}_V$, we set~$ \varphi = \Pi_V^\perp u$.

In the next lemma,  under hypotheses which implies the existence of a ground state~$u_V$ for~$H_V$, we obtain an equation satisfied by~$\varphi$ at a minimizer. Moreover, using the Euler-Lagrange equation for~$\vec{A}_{\mathrm{gs}}$, we obtain a control over~$u_\mathrm{gs}$ and $\vec{A}_\mathrm{gs}$.
\begin{lemma}\label{lm:lagrange-PF}
Suppose that $V$ satisfies Hypotheses~\ref{condVPauliFierz} and \ref{condGS} and that $\chi$ satisfies Hypothesis~\ref{condChi}.
Let $(u_\mathrm{gs}, \vec{A}_\mathrm{gs})$ in $\mathcal{U\times A}$ be a global minimizer of $\mathcal{E}$. 
Then
\begin{multline}
\varphi_\mathrm{gs} = R_{E_V}\Proj \Big[2(-i\vec{\nabla} u_\mathrm{gs})\cdot (g\hat{\chi}*\vec{A}_\mathrm{gs})-g\hat{\chi}*\vec{\sigma} \cdot (\vec{\nabla}\wedge\vec{A}_{\mathrm{gs}})u_{\mathrm{gs}} \\
-(g\hat{\chi}*\vec{A}_\mathrm{gs})^2u_\mathrm{gs} - \frac{1}{32\pi^3}\|\vec{A}_\mathrm{gs}\|_{\dot{H}^1}^2 u_\mathrm{gs} \Big] \,.\label{eq:PF-phi}
\end{multline}
Moreover, there exist $\varepsilon_V>0$ and $C_V>0$ such that, if
\begin{equation}\label{eq:small_chi}
\mathbf{g}_\chi:=|g|\Big\|{\frac{\chi}{\abs{k}}}\Big\|_{L^2+L^{3,\infty}}\le\varepsilon_V,
\end{equation}
then the following estimates hold
\begin{align}
& \norm{\varphi_{\mathrm{gs}}}_{\mathcal{Q}_V}  \le C_V\mathbf{g}_\chi^2 \,,\label{eqn:estimate_phi_gs}\\
& \big\|\vec{A}_{\mathrm{gs}}\big\|_{\dot{H}^1}\leq C_V\mathbf{g}_\chi\,, \label{eqn:estimate_A_gs_order_1}\\
& \big\|\vec{A}_{\mathrm{gs}} - \vec{A}_{\mathrm{gs}}^{[1]}  \big\|_{\dot{H}^1} = \big\|\vec{\nabla} \wedge \vec{A}_{\mathrm{gs}}- \vec{\nabla} \wedge \vec{A}_{\mathrm{gs}}^{[1]}\big\|_{L^2}\leq C_V\mathbf{g}_\chi^3\,, \label{eqn:expansion_A_gs_order_3}
\end{align}
with
\begin{equation}\label{eqn:A[1]}
\vec{A}_{\mathrm{gs}}^{[1]} :=  16\pi^3 (-\Delta)^{-1} g\hat{\chi} * \vec{\nabla}  \wedge u_V^2\vec{\omega}_{\mathrm{gs}}\,,
\end{equation}
and where the vector $\vec{\omega}_{\mathrm{gs}}$ in $\mathbb{C}^3$ is defined by the relation 
\begin{equation}\label{eq:defomega}
 u_V^2 \vec{\omega}_{\mathrm{gs}}=\left\langle \Pi_Vu_{\mathrm{gs}}, \vec{\sigma}\, \Pi_Vu_{\mathrm{gs}} \right\rangle_{\mathbb{C}^2}\,.
 \end{equation}
\end{lemma}
\begin{remark}\label{rk:omega}
Note that $\big| |\vec{\omega}_{\mathrm{gs}}|^2-1\big| \leq C_V \mathbf{g}_\chi^4$ by \eqref{eqn:estimate_phi_gs}.
\end{remark}
\begin{proof}[Proof of Lemma~\ref{lm:lagrange-PF}]
To prove \eqref{eq:PF-phi}, it suffices to observe that $E_V\leq \mu_V=\inf \sigma(H_V)$ (see~\eqref{eq:upper-bound-EV-1}), and hence that $R_{E_V}\Pi_V^\perp$ is well-defined and identifies with an element of $\LL(\mathcal{Q}_V^*,\mathcal{Q}_V)$. Applying $R_{E_V}\Proj$ to \eqref{eq:PF-HVA-uc} then yields \eqref{eq:PF-phi}.

Now we prove \eqref{eqn:estimate_A_gs_order_1}. First observe that, by Lemma \ref{lem:coercivity-pauli-fierz}, the assumption \eqref{eq:small_chi} and the fact that $E_V\le\mu_V$, we have, for~$\varepsilon_V$ sufficiently small,
\begin{equation*}
\big\|(u_{\mathrm{gs}},\vec{A}_{\mathrm{gs}})\big\|_{\mathcal{U}\times\mathcal{A}}\le C_V,
\end{equation*}
for some constant $C_V>0$, uniformly in $g$ and $\chi$ such that $\mathbf{g}_\chi\leq \varepsilon_V$. Using \eqref{eq:PF-Lagrange-A} and Lemmata~\ref{lem:estimateChiAU}--\ref{lem:estimateChikFu1u2}, we then estimate
\begin{align}
\big\|\vec{A}_{\mathrm{gs}}\big\|_{\dot{H}^1} &\le C \big\| \, g\hat{\chi}*\Real \langle -i\vec{\nabla} u_\mathrm{gs} + \vec{\nabla}\wedge\vec{\sigma}u_{\mathrm{gs}},u_{\mathrm{gs}}\rangle_{\mathbb{C}^2}\big\|_{\dot{H}^{-1}}
+ C \big\| \, g\hat{\chi}*[(g\hat{\chi}*\vec{A}_{\mathrm{gs}})\abs{u_{\mathrm{gs}}}^2_{\mathbb{C}^2}] \big\|_{\dot{H}^{-1}} \notag\\
&\le C |g|\norm{\frac{\chi}{\abs{k}}}_{L^2+L^{3,\infty}}\norm{u_{\mathrm{gs}}}_{H^1}
 +Cg^2\norm{\frac{\chi}{\abs{k}}}_{L^2+L^{3,\infty}}^2 \big\|\vec{A}_{\mathrm{gs}}\big\|_{\dot{H}^1} \norm{u_{\mathrm{gs}}}_{H^1} \notag \\
 & \leq 2C R_V \mathbf{g}_\chi \, .
 \label{eq:estim_normA}
\end{align}
Next using \eqref{eq:PF-phi}, the fact that $R_{E_V}\Proj\in\mathcal{L}(\mathcal{Q}_V^*,\mathcal{Q}_V)$, the continuous embedding $L^2\subset\mathcal{Q}_V^*$, \eqref{eq:estim_normA} and Lemma \ref{lem:estimateChiAU}, we  obtain
\begin{align*}
\norm{\varphi_{\mathrm{gs}}}_{\mathcal{Q}_V} \le C_V\mathbf{g}_\chi^2\,, %\label{eq:estim_normphi2}
\end{align*}
for some constant~$C_V>0$.

Let us now prove \eqref{eqn:expansion_A_gs_order_3} starting from the formula given in \eqref{eq:PF-Lagrange-A} for~$\vec{A}_{\mathrm{gs}}$. Note that the constant $C_V$ might change from one line to the other.
Applying Lemma~\ref{lem:estimateChikFu1u2} first and then Lemma~\ref{lem:estimateChiAU}, the boundedness of $u_{\mathrm{gs}}$ in $\mathcal{Q}_V$ and \eqref{eqn:estimate_A_gs_order_1},
\begin{multline}\label{eqn:estimate_A_gs_term_chiA}
\|(-\Delta)^{-1}g{\hat{\chi}}*\Real
\big\langle ( g\hat{\chi}*\vec{A}_\mathrm{gs}) u_\mathrm{gs},u_{\mathrm{gs}}\big\rangle_{\mathbb{C}^2}\|_{\dot{H}^1}
\leq C_V \mathbf{g}_\chi \| ( g\hat{\chi}*\vec{A}_\mathrm{gs}) u_\mathrm{gs} \|_{L^2}  \|u_{\mathrm{gs}}\|_{H^1}\\
\leq C_V \mathbf{g}_\chi^2 \| \vec{A}_\mathrm{gs}\|_{\dot{H}^1} \|u_\mathrm{gs} \|_{H^1}^2
\leq C_V \mathbf{g}_\chi^3 \,.
\end{multline}
By Lemma~\ref{lem:estimateChiAU}, the boundedness of $u_{\mathrm{gs}}$ in $\mathcal{Q}_V$ and \eqref{eqn:estimate_phi_gs}, 
\begin{multline}\label{eqn:estimate_A_gs_term_inabla}
\|(-\Delta)^{-1}g{\hat{\chi}}*\Real
\big\langle -i\vec{\nabla}  u_\mathrm{gs},u_{\mathrm{gs}}\big\rangle_{\mathbb{C}^2}\|_{\dot{H}^1} \\
\leq 
\|g{\hat{\chi}}*\Real
\big\langle i\vec{\nabla}  \Pi_V u_\mathrm{gs},\Pi_V u_{\mathrm{gs}}\big\rangle_{\mathbb{C}^2}\|_{\dot{H}^{-1}}
+2\|g{\hat{\chi}}*
\big\langle \vec{\nabla}  \Pi_V u_\mathrm{gs},\varphi_{\mathrm{gs}}\big\rangle_{\mathbb{C}^2}\|_{\dot{H}^{-1}} 
+\|g{\hat{\chi}}*
\big\langle \vec{\nabla}  \varphi_\mathrm{gs},\varphi_{\mathrm{gs}}\big\rangle_{\mathbb{C}^2}\|_{\dot{H}^{-1}} \\
\leq 
0+C_V \mathbf{g}_\chi \| \vec{\nabla}  \Pi_V u_\mathrm{gs}\|_{L^2}  \| \varphi_{\mathrm{gs}}\|_{\dot{H}^{1}}
+C_V \mathbf{g}_\chi \| \vec{\nabla}  \varphi_{\mathrm{gs}}\|_{L^2}  \| \varphi_{\mathrm{gs}}\|_{\dot{H}^{1}} 
\leq 
C_V \mathbf{g}_\chi^3  \,.
\end{multline}
Applying Lemma~\ref{lem:estimateChikFu1u2} first,  then the boundedness of $u_{\mathrm{gs}}$ in $\mathcal{Q}_V$ and \eqref{eqn:estimate_phi_gs}, we obtain
\begin{multline}\label{eqn:estimate_A_gs_term_nablawedge}
\|(-\Delta)^{-1}g{\hat{\chi}}*\Real
\big\langle \vec{\nabla}\wedge\vec{\sigma}\,u_\mathrm{gs},u_{\mathrm{gs}}\big\rangle_{\mathbb{C}^2}
-(-\Delta)^{-1}g{\hat{\chi}}*\Real
\big\langle \vec{\nabla}\wedge\vec{\sigma}\,\Pi_V u_\mathrm{gs}, \Pi_V u_{\mathrm{gs}}\big\rangle_{\mathbb{C}^2}\|_{\dot{H}^1} \\
\leq 2\| g{\hat{\chi}}*
\big\langle \vec{\nabla}\wedge\vec{\sigma}\, \varphi_{\mathrm{gs}} , \Pi_V u_\mathrm{gs}\big\rangle_{\mathbb{C}^2}\|_{\dot{H}^{-1}}  
+  \| g{\hat{\chi}}*
\big\langle \vec{\nabla}\wedge\vec{\sigma}\,\varphi_\mathrm{gs}, \varphi_{\mathrm{gs}}\big\rangle_{\mathbb{C}^2}\|_{\dot{H}^{-1}}  \\
\leq C_V \mathbf{g}_\chi \| \vec{\nabla}\wedge\vec{\sigma}\, \varphi_{\mathrm{gs}}\|_{L^2} ( \|\Pi_V u_\mathrm{gs}\|_{H^{1}}  + \|\varphi_\mathrm{gs}\|_{H^{1}}  ) 
\leq C_V \mathbf{g}_\chi^3 \,.
\end{multline}
Then, from $\vec{\nabla} \wedge \langle u, \vec{\sigma} u\rangle = 2\Real\langle u,\vec{\nabla} \wedge \vec{\sigma}u \rangle$ the equality
\begin{align*}
(-\Delta)^{-1}g{\hat{\chi}}*2\Real
\big\langle \vec{\nabla}\wedge\vec{\sigma}\,\Pi_V u_\mathrm{gs}, \Pi_V u_{\mathrm{gs}}\big\rangle_{\mathbb{C}^2}
=(-\Delta)^{-1} g\hat{\chi} * \vec{\nabla} u_V^2 \wedge \vec{\omega}_{\mathrm{gs}}
\end{align*}
follows, which, along with \eqref{eqn:estimate_A_gs_term_chiA}, \eqref{eqn:estimate_A_gs_term_inabla} and \eqref{eqn:estimate_A_gs_term_nablawedge}, yields \eqref{eqn:expansion_A_gs_order_3}.

Finally, since $\vec{\nabla}\cdot \vec{A}_{\mathrm{gs}}= \vec{\nabla}\cdot \vec{A}^{[1]}_{\mathrm{gs}} = 0$, using the formula $\vec{\nabla} \wedge \vec{\nabla} \wedge \vec{A}= -\Delta\vec{A} - \vec{\nabla} (\, \vec{\nabla}\cdot \vec{A})$,
$$\big\|\vec{\nabla} \wedge (\vec{A}_{\mathrm{gs}}-\vec{A}^{[1]}_{\mathrm{gs}})\big\|_{L^2} 
= \big\|\vec{A}_{\mathrm{gs}}-\vec{A}^{[1]}_{\mathrm{gs}}\big\|_{\dot{H}^1}\,.$$
This concludes the proof of Lemma~\ref{lm:lagrange-PF}.
\end{proof}
Now we can prove that~$\vec{A}_{\mathrm{gs}}$ is fully determined by~$u_{\mathrm{gs}}$ for small~$g$.
\begin{proposition}\label{prop:PF-uniqueness}
Suppose that $V$ satisfies Hypotheses~\ref{condVPauliFierz} and \ref{condGS} and that $\chi$ satisfies Hypothesis~\ref{condChi}. Suppose that the decomposition $V=V_1+V_2$ of Hypothesis~\ref{condVPauliFierz} can be chosen such that $E_{V_1}>E_V$. There exists $\varepsilon_V>0$ such that, if
\begin{equation*}
\mathbf{g}_\chi:=|g|\Big\|\frac{\chi}{|k|}\Big\|_{L^2+L^{3,\infty}}\le\varepsilon_V,
\end{equation*}
 then if $(u_\mathrm{gs} ,\vec{A}_\mathrm{gs})$ and $(u_\mathrm{gs} ,\vec{A}^\prime_\mathrm{gs})$ are minimizers of $\mathcal{E}$, necessarily $\vec{A}_\mathrm{gs}=\vec{A}^\prime_\mathrm{gs}$.
\end{proposition}
\begin{proof}
In this proof we drop the indices $\mathrm{gs}$ to simplify the notations. For sufficiently small~$\mathbf{g}_\chi$, if  $(u,\vec{A})$ and $(u,\vec{A}')$ are minimizers, \eqref{eq:PF-Lagrange-A} yields
\begin{align*}
\vec{A}-\vec{A}' = -32\pi^3(-\Delta)^{-1}g{\hat{\chi}}*[(g\hat{\chi}*(\vec{A}-\vec{A}'))\abs{u}^2_{\mathbb{C}^2}] \,.
\end{align*}
It follows from Lemmata~\ref{lem:estimateChikFu1u2} and~\ref{lem:estimateChiAU} that
\begin{align*}
\big\| \vec{A} -\vec{A}' \big\|_{\dot{H}^{1}}
&= 32\pi^3 \big\| g\hat{\chi}*[(g\hat{\chi}*(\vec{A} -\vec{A}' )) |u|^2_{\mathbb{C}^2}] \big\|_{\dot{H}^{-1}}\notag\\
&\le C  \mathbf{g}_\chi \big\|(g\hat{\chi}*(\vec{A} -\vec{A}' ))u\|_{L^2} \|u\|_{H^1}\notag\\
&\le C^2 \mathbf{g}_\chi^2 \big\|\vec{A} -\vec{A}'\big\|_{\dot{H}^1} \|u\|_{H^1}^2\\
&\le C_V^2 \mathbf{g}_\chi^2 \big\|\vec{A} -\vec{A}'\big\|_{\dot{H}^1} \,,
\end{align*}
which implies that, for sufficiently small~$\mathbf{g}_\chi$, $\vec{A}=\vec{A}'$.
\end{proof}

\subsection{Expansion of the minimum at small coupling}\label{sec:PF-expansion}
In this section, we prove Proposition \ref{prop:Asymtotic-Expansion-Ground-State-Energy}, by establishing the asymptotic expansion \eqref{expansion-pauli-fierz}.

\begin{proof}[Proof of Proposition \ref{prop:Asymtotic-Expansion-Ground-State-Energy}]
Recall that
\begin{multline*} 
E_V=\E_V(u_{\mathrm{gs}},\vec{A}_{\mathrm{gs}}) =\overset{(i)}{\overbrace{ \ps{u_{\mathrm{gs}}}{H_Vu_{\mathrm{gs}}} }} -2\overset{(ii)}{\overbrace{\Real\big\langle-i\vec{\nabla} u_{\mathrm{gs}},(g\hat{\chi}*\vec{A_{\mathrm{gs}}})u_{\mathrm{gs}}\big\rangle }} \\
\quad-  \underset{(iii)}{\underbrace{\big\langle u_{\mathrm{gs}},\vec{\sigma}\cdot (g\hat\chi * \vec{\nabla} \wedge \vec{A}_{\mathrm{gs}}) u_{\mathrm{gs}} \big\rangle}}
+\underset{(iv)}{\underbrace{\big\langle u_{\mathrm{gs}},(g\hat{\chi}*\vec{A_{\mathrm{gs}}})^2u_{\mathrm{gs}}\big\rangle}}
+\frac{1}{32\pi^3}\underset{(v)}{\underbrace{\big\|\vec{A}_{\mathrm{gs}}\big\|^2_{\dot{H}^1}}}\,. %\label{eq:recall_energy}
\end{multline*}

For~$(i)$, using that $H_V \Pi_V u_{\mathrm{gs}}= \mu_V \Pi_V u_{\mathrm{gs}}$, $\Pi_V u_{\mathrm{gs}} \perp \varphi_{\mathrm{gs}}$ and  \eqref{eqn:estimate_phi_gs}, we obtain
\begin{align*}%\label{eqn:expansion_E_V_(i)}
|(i) 
 - \mu_V|\leq |\ps{\varphi_{\mathrm{gs}}}{H_V\varphi_{\mathrm{gs}}}| \leq C_V \mathbf{g}_\chi^4. 
\end{align*}

Next we decompose $(ii)$ into three terms:
\begin{multline*}
(ii)
= g\Real\big\langle-i\vec{\nabla} \Pi_V u_{\mathrm{gs}},(\hat{\chi}*\vec{A}_{\mathrm{gs}}) \Pi_V u_{\mathrm{gs}}\big\rangle \\
+2g\Real\big\langle-i\vec{\nabla} \Pi_V u_{\mathrm{gs}},(\hat{\chi}*\vec{A}_{\mathrm{gs}})\varphi_{\mathrm{gs}}\big\rangle
+g\Real\big\langle-i\vec{\nabla} \varphi_{\mathrm{gs}},(\hat{\chi}*\vec{A}_{\mathrm{gs}}) \varphi_{\mathrm{gs}}\big\rangle.
\end{multline*}
The first term vanishes because, with $\Pi_V u_{\mathrm{gs}} = \big(\begin{smallmatrix}a\\
b
\end{smallmatrix}\big) u_V$ for some coefficients~$a$ and~$b$ in~$\mathbb{C}$,
\begin{equation*}%\label{eqn:energy_nabla_Pi_u}
g\Real\big\langle-i\vec{\nabla} \Pi_V u_{\mathrm{gs}},(\hat{\chi}*\vec{A}_{\mathrm{gs}}) \Pi_V u_{\mathrm{gs}}\big\rangle 
= g(|a|^2+|b|^2)\Real\big\langle-i\vec{\nabla} u_V,(\hat{\chi}*\vec{A}_{\mathrm{gs}}) u_V \big\rangle =0\,,
\end{equation*}
since $u_V$, $\hat\chi$ and $\vec{A}_{\mathrm{gs}}$ are real-valued.
The next term in $(ii)$ is controlled using Cauchy-Schwarz's inequality followed by Lemmata~\ref{lem:estimateChiAU} and~\ref{lm:lagrange-PF},
\begin{equation*}
|\big\langle-i\vec{\nabla} \Pi_V u_{\mathrm{gs}},(g\hat{\chi}*\vec{A}_{\mathrm{gs}})\varphi_{\mathrm{gs}}\big\rangle| \\
\leq C_V \mathbf{g}_\chi \|\vec{\nabla} \Pi_V u_{\mathrm{gs}}\|_{L^2} \| \vec{A}_{\mathrm{gs}} \|_{\dot{H}^1} \| \varphi_{\mathrm{gs}}\|_{H^1}
\leq C_V \mathbf{g}_\chi^4 \,.
\end{equation*}
The last term in $(ii)$ is bounded by~$C_V \mathbf{g}_\chi^6$ using similar arguments. 

Similarly, $(iv)$ is bounded by~$C_V \mathbf{g}_\chi^4$.

As for $(iii)$, using the formula $\int \vec{w}_1\cdot \vec{\nabla}\wedge \vec{w}_2 = \int \vec{w}_2\cdot \vec{\nabla}\wedge \vec{w}_1 $\,, we can rewrite
$$
(iii)
=\int (g\hat\chi *  \vec{A}_{\mathrm{gs}}) \cdot \vec{\nabla} \wedge \big\langle u_{\mathrm{gs}},\vec{\sigma} u_{\mathrm{gs}} \big\rangle_{\mathbb{C}^2} \,.
$$
In this form, it can be shown using the same arguments as before that 
\begin{align*}
\big|
(iii)
- \int (g\hat\chi *  \vec{A}^{[1]}_{\mathrm{gs}}) \cdot \vec{\nabla} \wedge u_V^2\vec{\omega}_{\mathrm{gs}} \big| 
\leq C_V \mathbf{g}_\chi^4,
\end{align*}
where we recall that $\omega_{\mathrm{gs}}$ has been defined in \eqref{eq:defomega}. Now, thanks to \eqref{eqn:A[1]} and the formula $\vec{\nabla} \wedge \vec{\nabla} \wedge u_V^2\vec{\omega}_{\mathrm{gs}}= -\Delta u_V^2\vec{\omega}_{\mathrm{gs}} - \vec{\nabla} (\, \vec{\nabla}\cdot u_V^2\vec{\omega}_{\mathrm{gs}})$, we have
\begin{multline}
\int (g\hat\chi *  \vec{A}^{[1]}_{\mathrm{gs}}) \cdot \vec{\nabla} \wedge u_V^2\vec{\omega}_{\mathrm{gs}}
=  16\pi^3  \int  (g\hat{\chi} * \vec{\nabla}  \wedge u_V^2\vec{\omega}_{\mathrm{gs}}) \cdot ((-\Delta)^{-1} g\hat{\chi} * \vec{\nabla}  \wedge u_V^2\vec{\omega}_{\mathrm{gs}}) \\
 = 16\pi^3 \int (g\hat\chi * u_V^2)^2 |\vec{\omega}_{\mathrm{gs}}|^2 + 16\pi^3 \int ((-\Delta)^{-1}g\hat\chi * \vec{\omega}_{\mathrm{gs}}\cdot \vec{\nabla} u_V^2)(g\hat\chi * \vec{\omega}_{\mathrm{gs}}\cdot \vec{\nabla} u_V^2) . \label{eq:351}
\end{multline}

To estimate $(v)$, we use Lemma~\ref{lm:lagrange-PF}, which shows that
\begin{equation*}%\label{eqn:expansion_E_V_(v)}
\big|(v)
-\|\vec{A}^{[1]}_{\mathrm{gs}}\|^2_{\dot{H}^1}\big| \leq C_V \mathbf{g}_\chi^4 \,. 
\end{equation*}
A direct computation then gives
\begin{align*}
\|\vec{A}^{[1]}_{\mathrm{gs}}\|^2_{\dot{H}^1} = (16\pi^3)^2  \int  (g\hat{\chi} * \vec{\nabla}  \wedge u_V^2\vec{\omega}_{\mathrm{gs}}) \cdot ((-\Delta)^{-1} g\hat{\chi} * \vec{\nabla}  \wedge u_V^2\vec{\omega}_{\mathrm{gs}}),
\end{align*}
namely we obtain the same term as in \eqref{eq:351}, with a different pre-factor.

Putting all together, we have shown that
\begin{equation*}
\Big | E_V -  \mu_V + 8\pi^3 g^2 \Big ( \int (\hat\chi * u_V^2)^2 |\vec{\omega}_{\mathrm{gs}}|^2 +  \int ((-\Delta)^{-1}\hat\chi * \vec{\omega}_{\mathrm{gs}}\cdot \vec{\nabla} u_V^2)(\hat\chi * \vec{\omega}_{\mathrm{gs}}\cdot \vec{\nabla} u_V^2) \Big ) \Big | \le C_V \mathbf{g}_\chi^4.
\end{equation*}
We have $|1-|\vec{\omega}_{\mathrm{gs}}|^2| \leq C_V \mathbf{g}_\chi^4$ (see Remark \ref{rk:omega}). Moreover, in the case of a radial potential~$V$, the ground state $u_V$ of $H_V$ is radial. If in addition $\chi$ is radial, then the second term in the right-hand side of the previous equation can be expressed independently of~$\vec{\omega}_{\mathrm{gs}}$. This directly leads to \eqref{expansion-pauli-fierz}.
\end{proof}

\subsection{Ultraviolet limit of the ground state energies}\label{sec:UV-PF}

We suppose in this section that the interaction between the non-relativistic particle and the field is cut-off in the ultraviolet, i.e.~that the Maxwell--Schr\"odinger energy functional is given by \eqref{eq:comput_E(u,A)Lambda_intro}
with $\chi_\Lambda=\chi\mathds{1}_{|k|\le\Lambda}$, for some ultraviolet parameter $\Lambda>0$. We then study the limit $\Lambda\to\infty$.
In this section we drop the index~$V$ for~$\mathcal{E}_V$ as the potential remains fixed throughout the section.

As in the previous sections, $\chi$ will be fixed such that $\chi/|k|$ lies in~$L^{2}+L^{3,\infty}$. In particular, we have $\chi_\Lambda/|k|$ in~$L^2$ and $\chi_\Lambda/\sqrt{|k|}$ in~$L^2$, which in turn implies that the ultraviolet cut-off Pauli-Fierz Hamiltonian 
\[
\mathbb{H}_\Lambda:= \big( \vec{\sigma}\cdot \big(-i\vec{\nabla}_{x}\otimes\mathbf{I}_{\text{f}}-\vec{\mathbb{A}}(\vec{m}_{\chi_\Lambda,x})\big)\big)^{2}+V\otimes\mathbf{I}_{\text{f}}+\mathbf{I}_{\text{el}}\otimes \mathbb{H}_{\text{f}},
\]
identifies to a self-adjoint operator (see Appendix \ref{app:Fock}).

We show that the ground state energies $E_{V,\Lambda}$ defined in \eqref{eq:GSenergy_Lambda} converge to $E_V$ in the ultraviolet limit $\Lambda\to\infty$. It should be noted that, in general, $|k|^{-1}\chi_\Lambda$ does \emph{not} converge to $|k|^{-1}\chi$ in $L^2+L^{3,\infty}$. To circumvent this difficulty, we will use the following relation
\begin{equation}\label{eq:rel_Lambda}
\hat{\chi}_\Lambda*\vec{A}=(2\pi)^{-3}\mathcal{F}(\chi_\Lambda\bar{\FF}(\vec{A}))=(2\pi)^{-3}\mathcal{F}(\chi\bar{\FF}(\mathds{1}_{|-i\vec{\nabla}|\le\Lambda}\vec{A}))=\hat\chi*\vec{A}_{\le\Lambda},
\end{equation}
where we have set
\begin{equation*}
\vec{A}_{\le\Lambda}:=\mathds{1}_{|-i\vec{\nabla}|\le\Lambda}(\vec{A}).
\end{equation*}
Recalling from \eqref{eq:def_space_A} that $\mathcal{A} = \{ \vec A\in \dot{H}^1(\mathbb{R}^3;\mathbb{R}^3) \mid  \vec \nabla \cdot \vec A = 0\}$, we introduce the subspace
 \begin{align}
\mathcal{A}_{\le\Lambda} := \{ \vec A\in\mathcal{A} \mid \vec{A}=\mathds{1}_{|-i\vec{\nabla}|\le\Lambda}(\vec{A}) \} \,. \label{eq:def_space_A_Lambda}
\end{align}
We then have the following identity.
\begin{lemma}\label{lm:Leb2}
Suppose that $V$ satisfies Hypothesis \ref{condVPauliFierz} and that $\chi$ satisfies Hypothesis \ref{condChi}. Then, for all $\Lambda>0$,
\begin{equation*}
E_{V,\Lambda}=\inf_{(u,\vec{A})\in\mathcal{U}\times\mathcal{A}_{\le\Lambda}}\mathcal{E}(u,\vec{A}).
\end{equation*}
\end{lemma}
\begin{proof}
It suffices to observe that, by \eqref{eq:rel_Lambda}, the following equality holds for all $(u,\vec{A})$ in~$\mathcal{U}\times\mathcal{A}$:
\begin{align}\label{eq:rel_Lambda2}
\mathcal{E}_\Lambda(u,\vec{A})=\mathcal{E}(u,\vec{A}_{\le\Lambda})+\frac{1}{32\pi^3}\big\| \mathds{1}_{|-i\vec{\nabla}|\ge\Lambda}(\vec A) \big\|^2_{\dot{H}^1}.
\end{align}
The statement of the lemma directly follows.
\end{proof}
If $\vec{A}$ belongs to~$\dot{H}^1$, we have that $\|\vec{A}_{\le\Lambda}-\vec{A}\|_{\dot{H}^1}\to 0$ as $\Lambda\to\infty$.
Now we can prove the convergence of the ground state energies in the ultraviolet limit. 
\begin{proof}[Proof of Proposition \ref{prop:conv-GS-en-PF_intro}]
For $0<\Lambda\le\Lambda'$, we have $\mathcal{A}_{\le\Lambda}\subseteq\mathcal{A}_{\le\Lambda'}\subset\mathcal{A}$ and hence, by Lemma~\ref{lm:Leb2},
\begin{align*}
E_{V}\le E_{V,\Lambda'}\le E_{V,\Lambda}.
\end{align*}
Therefore $\Lambda\mapsto E_{V,\Lambda}$ is non-increasing on $(0,\infty)$, bounded below by $E_{V}$, so that we can define
\begin{equation*}
E_{V,\infty}:=\lim_{\Lambda\to\infty}E_{V,\Lambda} \ge E_{V}.
\end{equation*}

To show that $E_{V,\infty}\le E_{V}$, let $\varepsilon>0$ and let $(u_\varepsilon,\vec{A}_\varepsilon)$ in~$\mathcal{U}\times\mathcal{A}$ be such that~$\mathcal{E}(u_\varepsilon,\vec{A}_\varepsilon)\le E_{V}+\varepsilon$. Using \eqref{eq:rel_Lambda2}, we have
\begin{align*}
E_{V,\Lambda}&\le \mathcal{E}_\Lambda(u_\varepsilon,\vec{A}_\varepsilon) = \mathcal{E}(u_\varepsilon,\vec{A}_{\varepsilon,\le\Lambda}) +\frac{1}{32\pi^3}\big\| \mathds{1}_{|-i\vec{\nabla}|\ge\Lambda}\vec{A}_{\varepsilon} \big\|^2_{\dot{H}^1} \\
&\le E_V+\varepsilon + \mathcal{E}(u_\varepsilon,\vec{A}_{\varepsilon,\le\Lambda}) - \mathcal{E}(u_\varepsilon,\vec{A}_{\varepsilon}) +\frac{1}{32\pi^3}\big\| \mathds{1}_{|-i\vec{\nabla}|\ge\Lambda}\vec{A}_{\varepsilon} \big\|^2_{\dot{H}^1}.
\end{align*}
A direct computation shows that
\begin{multline}
\Big|\mathcal{E}(u_\varepsilon,\vec{A}_{\varepsilon,\le\Lambda}) - \mathcal{E}(u_\varepsilon,\vec{A}_{\varepsilon}) +\frac{1}{32\pi^3}\big\| \mathds{1}_{|-i\vec{\nabla}|\ge\Lambda}\vec{A}_{\varepsilon} \big\|^2_{\dot{H}^1} \Big| \\
\le2|g| \, \big| \big\langle \vec{\sigma}\cdot(-i\vec{\nabla}u_\varepsilon),\vec{\sigma}\cdot\hat{\chi}*(\vec{A}_{\varepsilon,\le\Lambda}-\vec{A}_\varepsilon)u_\varepsilon\big\rangle_{L^2}\big| \\
\quad+g^2\Big|\big\|(\hat{\chi}*\vec{A}_{\varepsilon,\le\Lambda})u_\varepsilon\big\|_{L^2}^2-\big\|(\hat{\chi}*\vec{A}_{\varepsilon})u_\varepsilon\big\|_{L^2}^2\Big|. \label{eq:UV1a}
\end{multline}
Applying Lemma \ref{lem:estimateChiAU} gives
\begin{equation*}
\big| \big\langle \vec{\sigma}\cdot(-i\vec{\nabla}u_\varepsilon),\vec{\sigma}\cdot\hat{\chi}*(\vec{A}_{\varepsilon,\le\Lambda}-\vec{A}_\varepsilon)u_\varepsilon\big\rangle_{L^2}\big|\le C_{\chi,u_\varepsilon} \big\|\vec{A}_{\varepsilon,\le\Lambda}-\vec{A}_\varepsilon\big\|_{\dot{H}^1},
\end{equation*}
for some constant $C_{\chi,u_\varepsilon}$ depending on $\chi$ and $u_\varepsilon$. Likewise,
\begin{multline*}
\Big|\big\|(\hat{\chi}*\vec{A}_{\varepsilon,\le\Lambda})u_\varepsilon\big\|^2-\big\|(\hat{\chi}*\vec{A}_{\varepsilon})u_\varepsilon\big\|_{L^2}^2\Big|\\
\le\big\|\big(\hat{\chi}*(\vec{A}_{\varepsilon,\le\Lambda}-\vec{A}_\varepsilon)\big)u_\varepsilon\big\|_{L^2} \big(\big\|(\hat{\chi}*\vec{A}_{\varepsilon,\le\Lambda})u_\varepsilon\big\|_{L^2}+\big\|(\hat{\chi}*\vec{A}_{\varepsilon})u_\varepsilon\big\|_{L^2}\big)\\
\le C_{\chi,u_\varepsilon,\vec{A}_\varepsilon}\big\|\vec{A}_{\varepsilon,\le\Lambda}-\vec{A}_\varepsilon\big\|_{\dot{H}^1}.
\end{multline*}
Since $\|\vec{A}_{\varepsilon,\le\Lambda}-\vec{A}_\varepsilon\|_{\dot{H}^1}\to 0$ as $\Lambda\to\infty$, inserting the previous estimates into \eqref{eq:UV1a} and letting~$\Lambda\to\infty$, we obtain
\begin{align*}
E_{V,\infty}\le E_V+\varepsilon.
\end{align*}
Since $\varepsilon>0$ is arbitrary, this concludes the proof of the proposition.
\end{proof}

\appendix

\section{Operators in Fock space, self-adjointness}\label{app:Fock}
In this appendix we set up some notations and give the proof of the self-adjointness of the Pauli-Fierz operator $\mathbb{H}$. We recall the definitions of standard operators in Fock space. We only give here formal definitions, referring the reader to e.g. \cite{DerezinskiGerard99,ReedSimonII} for more details.

Let us consider a Hilbert space $\mathfrak{h}$ and its associated symmetric Fock space $\F(\mathfrak{h}):=\bigoplus_{n=0}^\infty\bigvee^n\mathfrak{h}$, with $\bigvee^0\mathfrak{h}:=\mathbb{C}$.  Let~$h\in\mathfrak{h}$. For $n\in\N,$ the creation and annihilation operators are respectively defined as 
\begin{align*}
a^*(h)_{|\bigvee^n\mathfrak{h}}&=\sqrt{(n+1)}\;|h\rangle\bigvee\I_{\bigvee^n\mathfrak{h}},\quad
a(h)_{|\bigvee^n\mathfrak{h}}=\sqrt{n}\;\langle h|\otimes\I_{\bigvee^{n-1}\mathfrak{h}}, \quad a(h)_{|\mathbb{C}}=0 \, .
\end{align*}
The field operator $\Phi(h)$ is then defined as
\begin{equation}\label{eqn:def-field-operator}
\Phi(h)=(a(h)+a^*(h))/\sqrt{2}.
\end{equation}

Let $\omega$ be a self-adjoint operator on $\mathfrak{h}.$ The second quantization of $\omega$ is the operator on Fock space defined by
\begin{equation*}
\mathrm{d}\Gamma(\omega)_{|\bigvee^n\mathfrak{h}}=\sum_{k=1}^n\I_{\bigvee^{k-1}\mathfrak{h}}\otimes \omega\otimes \I_{\bigvee^{n-k}\mathfrak{h}},\quad \mathrm{d}\Gamma(\omega)_{|\mathbb{C}}=0.
\end{equation*} 

The coherent state of parameter $f\in\mathfrak{h}$ is defined as
\begin{equation*}
\Psi_f:=e^{i\Phi\left(\frac{\sqrt{2}}{i}f\right)}\Omega=e^{-\frac{\norm{f}^2_{\mathfrak{h}}}{2}}\sum_{n=0}^\infty\frac{f^{\otimes n}}{\sqrt{n!}},
\end{equation*}
where $\Omega$ stands for the Fock vacuum. Coherent states are eigenvectors of the annihilation operators in the sense that for all $f,h\in\mathfrak{h},$  
\begin{equation*}
a(h)\Psi_f=\ps{h}{f}_{\mathfrak{h}}\Psi_f\,.
\end{equation*}
This in turn leads to the following equalities:
\begin{align}\label{eq:coher-app}
\ps{\Psi_f}{\Phi(h)\Psi_f}_{\F(\mathfrak{h})}&=2\Real\ps{h}{f}_\mathfrak{h},\quad \ps{\Psi_f}{\mathrm{d}\Gamma(\omega)\Psi_f}_{\F(\mathfrak{h})}=\ps{f}{\omega\, f}_\mathfrak{h}\,.
\end{align}

We recall the following estimates, which holds for any non-negative operator $\omega$ on $\mathfrak{h}$, $h$ in the domain of $\omega^{-1/2}$ and $\Psi$ in the domain of $\mathrm{d}\Gamma(\omega)^{1/2}$:
\begin{align}
&\| a(h)\Psi\|\le\|\omega^{-\frac12}h\|^2_{\mathfrak{h}}\|\mathrm{d}\Gamma(\omega)^{\frac12}\Psi\|^2_{\F(\mathfrak{h})},\label{eq:estim1}\\
& \| a^*(h)\Psi\|\le\|\omega^{-\frac12}h\|^2_{\mathfrak{h}}\|\mathrm{d}\Gamma(\omega)^{\frac12}\Psi\|^2_{\F(\mathfrak{h})}+\|f\|^2_{\mathfrak{h}}\|\Psi\|^2_{\F(\mathfrak{h})}.\label{eq:estim2}
\end{align}

The next proposition establishes the self-adjointness of the Pauli-Fierz Hamiltonian $\mathbb{H}$ of the standard model of non-relativistic QED (see \eqref{eq:defH}) under our assumptions. We recall a proof for the convenience of the reader.

\begin{proposition}[Self-adjointness of $\mathbb{H}$]\label{prop:self-adj}
Suppose that $V=V_+-V_-$ belongs to $L_{\mathrm{loc}}^{1}(\mathbb{R}^{3};\mathbb{R}^{+})$ and that $V_{-}$ is infinitesimally small with respect to $-\Delta$ in the sense of quadratic forms on $H^{1}(\mathbb{R}^3)$. Suppose in addition that~$\chi:\mathbb{R}^3\to\mathbb{R}$ is such that both $|k|^{-1/2}\chi$ and~$|k|^{-1}\chi$ are in~$L^2(\mathbb{R}^3).$ Then the Pauli-Fierz hamiltonian~$\mathbb{H}$ is a self-adjoint operator with form domain 
\begin{equation*}
\mathcal{Q}(\mathbb{H})=\mathcal{Q}\left(H_{V_+}\otimes\I_{\C^2}\otimes\I_\mathrm{f}+\I_\mathrm{el}\otimes\I_{\C^2}\otimes\mathrm{d}\Gamma(\abs{k})\right)\,.
\end{equation*}
\end{proposition}

\begin{proof}
Let $\mathcal{Q_+}:=\mathcal{Q}\left(H_{V_+}\otimes\I_{\C^2}\otimes\I_\mathrm{f}+\I_\mathrm{el}\otimes\I_{\C^2}\otimes\mathrm{d}\Gamma(\abs{k})\right).$ 
We claim that the following form is closed on $\mathcal{Q_+}$:
\begin{multline*}
q_+(\Psi_1,\Psi_2):=\ps{\vec{\sigma}\cdot(-i\vec{\nabla}\otimes\I_{\C^2}\otimes\I_\mathrm{f}-\vec{\mathbb{A}}(\vec{m}_x))\Psi_1}{\vec{\sigma}\cdot(-i\vec{\nabla}\otimes\I_{\C^2}\otimes\I_\mathrm{f}-\vec{\mathbb{A}}(\vec{m}_x))\Psi_2}_{\HHH}\\
+\ps{(V_+^{1/2}\otimes\I_{\C^2}\otimes\I_\mathrm{f})\Psi_1}{(V_+^{1/2}\otimes\I_{\C^2}\otimes\I_\mathrm{f}) \Psi_2}_{\HHH}\\
+\ps{(\I_\mathrm{el}\otimes\I_{\C^2}\otimes\mathbb{H}_f^{1/2})\Psi_1}{(\I_\mathrm{el}\otimes\I_{\C^2}\otimes\mathbb{H}_f^{1/2})\Psi_2}_{\HHH}\,.
\end{multline*}
Let $(\Psi_n)_{n\in\N}$ in $\mathcal{Q}_+^\N$ be such that $\Psi_n\underset{n\rightarrow\infty}{\overset{\HHH}{\longrightarrow}}\Psi$ and $(q_+(\Psi_n, \Psi_n))_{n\in\N}$ converges. We show that $\Psi$ is in $\mathcal{Q}_+$ and that $q_+(\Psi_n-\Psi,\Psi_n-\Psi)\underset{n\rightarrow\infty}{\longrightarrow}0.$ 
Considering the closed form on $\mathcal{Q}_+$ defined by
\begin{multline*}
q_0(\Psi_1, \Psi_2):=\ps{(-i\vec{\nabla}\otimes\I_{\C^2}\otimes\I_\mathrm{f}) \Psi_1}{(-i\vec{\nabla}\otimes\I_{\C^2}\otimes\I_\mathrm{f}) \Psi_2}_{\HHH}\\+\ps{(V_+^{1/2}\otimes\I_{\C^2}\otimes\I_\mathrm{f}) \Psi_1}{(V_+^{1/2}\otimes\I_{\C^2}\otimes\I_\mathrm{f}) \Psi_2}_{\HHH}\\+\ps{(\I_\mathrm{el}\otimes\I_{\C^2}\otimes\mathbb{H}_f^{1/2})\Psi_1}{(\I_\mathrm{el}\otimes\I_{\C^2}\otimes\mathbb{H}_f^{1/2})\Psi_2}_{\HHH}\,,
\end{multline*}
it is not difficult to verify, using \eqref{eq:estim1}--\eqref{eq:estim2}, that there exists a positive constant $C$ such that, for all $\varepsilon>0$, for all $\Psi$ in $\mathcal{Q}_+$,
\begin{equation}\label{eq:last}
(1-C\varepsilon)q_0(\Psi,\Psi)\leq q_+(\Psi,\Psi)+C\varepsilon\left(\norm{(\I_\mathrm{el}\otimes\I_{\C^2}\otimes\mathbb{H}_f^{1/2})\Psi}_{\HHH}^2+\norm{\Psi}_{\HHH}^2\right)\,.
\end{equation}
This implies that $(q_0(\Psi_n,\Psi_n))_{n\in\N}$ is a Cauchy sequence and hence, since $q_0$ is closed on~$\mathcal{Q}_+$, we deduce that $\Psi\in\mathcal{Q}_+.$ Finally, since, similarly as for \eqref{eq:last}, we have
\begin{equation*}
q_+(\Psi,\Psi) \le (1-C\varepsilon)q_0(\Psi,\Psi),
\end{equation*}
we conclude that
\begin{equation*}
q_+(\Psi_n-\Psi, \Psi_n-\Psi)\underset{n\rightarrow\infty}{\longrightarrow}0.
\end{equation*}

We have shown that the quadratic form~$q_+$ is positive and closed on $\mathcal{Q}_+.$  In particular (see e.g.~\cite[Chap.~6, Thm.~2.1]{Kato80}), this ensures that there exists a positive self-adjoint operator $H_+$ such that, for all $\Psi_1,\Psi_2\in\mathcal{Q}_+,$ $q_+(\Psi_1,\Psi_2)=\langle H_+^{1/2}\Psi_1, H_+^{1/2}\Psi_2\rangle_{\HHH}.$ Since $V_-$ is infinitesimally form bounded with respect to $-\Delta$, one easily deduces that it is also infinitesimally form bounded with respect to $H_+$. The KLMN theorem then allows us to conclude that $\mathbb{H}$ identifies to a self-adjoint operator with form domain
\begin{equation*}
\mathcal{Q}(\mathbb{H})=\mathcal{Q}_+=\mathcal{Q}\left(H_{V_+}\otimes\I_{\C^2}\otimes\I_\mathrm{f}+\I_\mathrm{el}\otimes\I_{\C^2}\otimes\mathrm{d}\Gamma(\abs{k})\right) \,.
\end{equation*}
This concludes the proof.
\end{proof}

\bibliographystyle{plain}
\bibliography{biblio}
\end{document}